\documentclass[twoside,11pt,reqno]{amsart}
\usepackage{amsmath,amssymb,amscd,mathrsfs,amscd, todonotes,amsthm}
\usepackage{graphics,verbatim}
\usepackage{todonotes}
\usepackage{hyperref}
\usepackage{enumitem}
\usepackage{stmaryrd}
\usepackage{tikz-cd}
\usetikzlibrary{matrix,arrows,decorations.pathmorphing}

\newcommand{\Ad}{\textup{Ad}}
\newcommand{\ad}{\textup{ad}}
\newcommand{\Zp}{\mathbb{Z}_{(p)}}
\newcommand{\Ng}{\mathcal{N}(\mathfrak{g})}

\newcommand{\g}{\mathfrak{g}}
\newcommand{\gZp}{\mathfrak{g}_{\mathbb{Z}_{(p)}}}

\newcommand{\Fp}{\mathbb{F}_p}
\newcommand{\gFp}{\mathfrak{g}_{\mathbb{F}_p}}

\newcommand{\gC}{\mathfrak{g}_{\mathbb{C}}}
\newcommand{\gF}{\mathfrak{g}_{\mathbb{F}}}
\newcommand{\gR}{\mathfrak{g}_{R}}
\newcommand{\gZ}{\mathfrak{g}_{\mathbb{Z}}}

\makeatletter
\newcommand{\leqnomode}{\tagsleft@true}
\newcommand{\reqnomode}{\tagsleft@false}
\makeatother

\newtheorem{theorem}{Theorem}[subsection]

\makeatletter\let\c@fact\c@theorem\makeatother

\makeatletter\let\c@note\c@theorem\makeatother

\newtheorem{lemma}{Lemma}[subsection]
\makeatletter\let\c@lemma\c@theorem\makeatother

\makeatletter\let\c@lemma\c@theorem\makeatother

\makeatletter\let\c@alg\c@theorem\makeatother

\newtheorem{prop}{Proposition}[subsection]
\makeatletter\let\c@prop\c@theorem\makeatother

\makeatletter\let\c@conj\c@theorem\makeatother

\newtheorem{cor}{Corollary}[subsection]
\makeatletter\let\c@cor\c@theorem\makeatother

\newtheorem{defn}{Definition}[subsection]
\makeatletter\let\c@defn\c@theorem\makeatother

\theoremstyle{definition}
\newtheorem{example}{Example}[subsection]
\newtheorem{remark}{Remark}[subsection]
\makeatletter\let\c@remark\c@theorem\makeatother

\makeatletter\let\c@example\c@theorem\makeatother
\numberwithin{equation}{subsection}

\usepackage[capitalise]{cleveref}
\crefname{theorem}{Theorem}{Theorems}
\crefname{fact}{Fact}{Facts}
\crefname{note}{Note}{Notes}
\crefname{lemma}{Lemma}{Lemmas}
\crefname{alg}{Algorithm}{Algorithms}
\crefname{remark}{Remark}{Remarks}
\crefname{example}{Example}{Examples}
\crefname{prop}{Proposition}{Propositions}
\crefname{conj}{Conjecture}{Conjectures}
\crefname{cor}{Corollary}{Corollaries}
\crefname{defn}{Definition}{Definitions}
\crefname{equation}{\!\!}{\!\!} 

\newcounter{listequation}

\begin{document}
\title{Unipotent Elements and Generalized Exponential Maps}

\author{Paul Sobaje}
\address{Department of Mathematics \\
          University of Georgia \\
          Athens, GA 30602}
\email{sobaje@uga.edu}
\date{\today}

\begin{abstract}
Let $G$ be a simple and simply connected algebraic group over an algebraically closed field $\Bbbk$ of characteristic $p>0$.  Assume that $p$ is good for the root system of $G$ and that the covering map $G_{sc} \rightarrow G$ is separable.  In previous work we proved the existence of a (not necessarily unique) Springer isomorphism for $G$ that behaved like the exponential map on the resticted nullcone of $G$.

In the present paper we give a formal definition of these maps, which we call `generalized exponential maps.'  We provide an explicit and uniform construction of such maps for all root systems, demonstrate their existence over $\Zp$, and give a complete parameterization of all such maps.  One application is that this gives a uniform approach to dealing with the ``saturation problem'' for a unipotent element $u$ in $G$, providing a new proof of the known result that $u$ lies inside a subgroup of $C_G(u)$ that is isomorphic to a truncated Witt group.  We also develop a number of other explicit and new computations for $\g$ and for $G$.  This paper grew out of an attempt to answer a series of questions posed to us by P. Deligne, who also contributed several of the new ideas that appear here.
\end{abstract}

\maketitle

\section{Introduction}

\noindent \textit{Notation.}  In what follows, $G$ denotes a simple algebraic group over an algebraically closed field $\Bbbk$ of characteristic $p>0$.  Its Lie algebra $\mathfrak{g}$ has a $[p]$-th power map sending $X \in \g$ to the element we denote as $X^{[p]}$.  The nilpotent variety of $\g$ is $\mathcal{N}(\g)$, the unipotent variety of $G$ is $\mathcal{U}(G)$.  There is a split affine group scheme $G_{\mathbb{Z}}$ over $\mathbb{Z}$ that base changes over $\Bbbk$ to $G$.  It has Lie algebra $\gZ$, and for any commutative ring $R$ we obtain a corresponding affine group scheme $G_R$ with Lie algebra $\gR$.  We set $B_{\mathbb{Z}}$ to be a split Borel subgroup defined over $\mathbb{Z}$, $U_{\mathbb{Z}}$ is its unipotent radical (see Section \ref{chevalley} for more on these conventions).  The Coxeter number of $G$ is $h$.

\subsection{} A classical result by Springer \cite{Sp3} is that if $G$ is simply connected and the characteristic of $\Bbbk$ is good for $G$, then there exist $G$-equivariant homeomorphisms between $\Ng$ and $\mathcal{U}(G)$.  Such maps, later found under these assumptions to be isomorphisms of varieties, are known as `Springer isomorphisms.'  They also exist for any simple group $G$ (not necessarily simply connected) so long as the surjective morphism $G_{sc} \rightarrow G$ from the simply-connected cover is separable.  In such a case, we follow Pevtsova and Stark in calling $p$ a \textit{separably good prime} for the root datum of $G$ \cite{PS}.  The only situation in which good does not always imply separably good is in type $A$.

For $G_{\mathbb{C}}$, a Springer isomorphism is easily obtained by the exponential map.  Springer isomorphisms in separably good characteristic can then be viewed as the existence of a partial replacement for the exponential map, the replacement sharing certain properties enjoyed by the exponential for $G_{\mathbb{C}}$.  In Springer's original work, it was the $G_{\mathbb{C}}$-equivariance and bijectivity of the exponential that were desired.

We considered in \cite{So} the problem of finding Springer isomorphisms for $G$ that were ``even more like'' the exponential map, with our focus on what these isomorphisms should do upon restriction to the subvariety $\mathcal{N}_1(\g) \subseteq \Ng$ of all $X$ such that $X^{[p]}=0$.  In particular, we proved the existence of Springer isomorphisms that, when restricted to the unipotent radicals of certain parabolic subgroups of $G$, agreed with the isomorphism coming from the exponential map in characteristic $0$ (see \cite{Ser}).  At the time, our interest in $\mathcal{N}_1(\g)$ was due to its role in the theory of support varieties for finite subgroup schemes of $G$ (see \cite{SFB},\cite{CLN}, and \cite{F}).

This paper began as a response to a series of questions we received from Pierre Deligne regarding the functoriality and uniqueness of the isomorphisms we constructed in \cite{So}.  Answers to these questions were not clear, in part due to the fact that the isomorphisms given in  \textit{loc. cit.} for groups of exceptional type were shown to exist via indirect arguments, and relied on a result by Seitz on `saturating' a unipotent element in $G$ inside a connected abelian unipotent subgroup of $G$ (see \ref{saturate}).  Consequently, issues about subfields of definition, and other important algebro-geometric properties, were not immediately evident from our approach.

These issues, and much more, are resolved in this paper, which significantly improves upon the results in \cite{So}, and contains new techniques for moving back and forth between characteristics $0$ and $p$.  Some of the clearest and best insights here have been contributed by Deligne, who has graciously allowed us to include several of his arguments along with our own (see ``Acknowledgments,'' Remark \ref{Delignelifting}, and the beginning of Section \ref{modified}).

\subsection{Statement of main result} Let $\mathcal{W}_m$ denote the algebraic group of truncated Witt vectors of length $m$ over $\Bbbk$.  As a variety, $\mathcal{W}_m \cong \mathbb{A}^m$.  A Springer isomorphism $\phi$ is defined to be a \textit{generalized exponential map} if for each $m>0$, and each $X \in \Ng$ such that $m$ is the minimal integer with $X^{[p^m]}=0$, the map
$$(a_0,\ldots,a_{m-1}) \mapsto \phi(a_0X)\phi(a_1X^{[p]})\cdots \phi(a_{m-1}X^{[p^{m-1}]})$$
defines a closed embedding of algebraic groups $\mathcal{W}_m \rightarrow G$ (this is Definition \ref{embeddingproperty}).

This definition is motivated by the well-known method of embedding truncated Witt groups into $GL_n$ with the Artin-Hasse exponential series (see Section \ref{Artinforclassical}).   In fact, we used the Artin-Hasse exponential series in \cite{So} to obtain generalized exponential maps (not yet formally named as such) for groups of type $A-D$.  The key new idea here is that this can be generalized in a natural way to groups of exceptional type.  We prove the following theorems, the second of which is the main theorem of this paper.

Let $G$ be one of the following groups: $SL_n, SO_n, Sp_{2n}$, or an adjoint group with a root system of exceptional type.  Assume that $p$ is good for the root system of $G$.

\vspace{0.15in}
\noindent \textbf{Theorem A.}  \textit{The $[p]$-mapping on $\g$ lifts to a degree $p$, $G_{\Zp}$-equivariant morphism of affine schemes
$$m^p:\gZp \rightarrow \gZp,$$
having the property that if $i$ is such that $p^i \ge h$, then $(m^p)^i(\mathfrak{u}_{\Zp})=\{0\}$.}
\vspace{0.15in}

The morphism $m^p$ is established in Proposition \ref{allbasechangeswork}(e), and an argument for the vanishing of $(m^p)^i$ can be found in Section \ref{modified} (where one also finds a proof of the next result).  The construction given makes it clear that the morphism $m^p$ satisfies many desirable properties.  For classical groups, their defining representation (over $\Zp$), with ordinary $p$-th power matrix multiplication, provides this map ($p > 2$ for types $B-D$).  For groups of exceptional type, the adjoint representation gives a natural generalization of this construction.  The key is that the non-degeneracy of the trace form mod $p$ gives rise to a $G_{\Zp}$-equivariant module splitting of $\mathfrak{gl}(\gZp)$, so one now takes $p$-th power matrix multiplication followed by a projection map.  We note that for classical groups, if $V_{\Zp}$ is the natural $\Zp$-module for $G_{\Zp}$, then one similary obtains a splitting of $\mathfrak{gl}(V_{\Zp})$, only here the projection with respect to this splitting is the identity on odd powers of elements in $\gZp$, and is zero on even powers (here we are identifying $\gZp$ with its embedding in $\mathfrak{gl}(\gZp)$).

Lifting the $[p]$-mapping to $\gZp$ allows for a very straightforward generalization of the Artin-Hasse exponential morphism to groups of exceptional type.  Let $m^{p^i}$ denote the morphism $(m^p)^i$.

\vspace{0.15in}
\noindent \textbf{Theorem B.}  \textit{The map
$$\tilde{E}_p(X)=\exp\left(X + \frac{m^p(X)}{p} + \frac{m^{p^2}(X)}{p^2} + \cdots\right)$$
defines a Springer isomorphism for $G_{\mathbb{C}}$.  It restricts to a $B_{\mathbb{C}}$-equivariant isomorphism
$$\mathfrak{u}_{\mathbb{C}} \xrightarrow{\sim} U_{\mathbb{C}}$$
that is defined over $\Zp$.  Over $\Bbbk$, this isomorphism from $\mathfrak{u}$ to $U$ extends to a generalized exponential map for $G$.}
\vspace{0.15in}

\subsection{Other Results and Relation to Previous Work}

\subsubsection{Saturation}\label{saturate} The ``saturation problem'' deals with finding a way to embed a $p$-unipotent element $u \in G$ inside a subgroup of $G$ isomorphic to the additive group $\mathbb{G}_a$.   This terminology goes back to work of Serre in \cite{Ser1}, where saturation was used to obtain semisimplicity results in characteristic $p$ for tensor products of sufficiently small representations of an algebraic or abstract group.  Deligne extended Serre's work on semisimplicity to affine group schemes in \cite{D}.

Testerman proved that in good characteristic, $p$-unipotent elements could in fact be embedded inside a simple subgroup of $G$ of type $A_1$ \cite{T}.  Seitz strengthened this, proving that they are contained in `good $A_1$' subgroups, any two of which are conjugate by an element in $C_G(u)$, and used this to establish a canonical saturation of $u$.

Proud \cite{Pr} proved that saturation could be achieved for $u$ having order $p^m$ for any $m\ge 1$, with the embedding now in a subgroup isomorphic to $\mathcal{W}_m$, but he did not claim that the embedding was canonical.  Seitz \cite{Sei2} improved this result, showing that such subgroups could be found inside of $C_G(u)$, getting a little closer to finding a canonical saturation, though it was not clear that the saturatated subgroups he gave were isomorphic to a truncated Witt group rather than merely isogenous to one.

The results in this paper bring to completion the idea of McNinch in \cite[\S 7.5]{M}, showing in all types that the Artin-Hasse exponential (or a suitable generalization thereof) can be used to saturate unipotent elements of arbitrary order.  In \cite{So}, we used the work of Seitz to prove that generalized exponential maps existed in exceptional types.  The work presented here is independent of that result, so indeed gives a new proof of saturation.  Furthermore, the saturated subgroups are isomorphic to truncated Witt groups.

\subsubsection{Decomposing Centralizers} Using generalized exponential maps, we show that for any nilpotent element $X \in \g$, the closed subgroup $Z(C_G(X))^0 \le G$ decomposes, as an algebraic group, into a direct product of truncated Witt groups.  This strengthens a result by Seitz, who observed in \cite{Sei2} that such a decomposition holds as abstract groups.

\subsubsection{Parameterizing Springer Isomorphisms/Explicit Results For $\g$} We give a full parameterization of all generalized exponential maps for $G$.  We also give a more detailed description of the variety of all Springer isomorphisms for $G$.  An aspect of this is to construct equivariant variety endomorphisms on $\Ng$ that behave somewhat like multiplicative power maps on $\mathcal{N}(\mathfrak{gl}_n)$.  In some cases these maps can be lifted to $\Zp$, thanks to interesting multilinear maps that can be shown to exist from ${\gZp}^{\times i}$ to $\gZp$.

We believe these various explicit results will prove useful in the study of $G$ and other related contexts, and indicate in Remark \ref{GGGR} one connection to Taylor's work on Generalized Gelfand-Graev representations.

\subsubsection{On Good vs. Separably Good Characteristic} Springer isomorphisms cannot exist in bad characteristic due to the fact that the number of nilpotent and unipotent orbits do not match up in this case.  In type $A$, all primes are good, but the separability of the covering map is still needed to guarantee that Springer isomorphisms exist.  We look at some issues related to this problem, and give a detailed example of what goes wrong for $PGL_2$ in characteristic $2$.

\subsection{Acknowledgements}  Professor Deligne worked out his own answers to the questions he asked, proving independently some of what appears here.  We have adapted our own proofs in light of his, opting for the clearest and strongest results, and this paper is significantly better because of the ideas he has shared.  We are very grateful to Professor Deligne for allowing us to include his work, for helpful conversations about this material, and for his interest in \cite{So}.

We also thank Jean-Pierre Serre for suggesting that we address what happens when $p$ is good but not separably good (such as $PGL_2$ in characteristic $2$), and for communicating to us some observations as to what goes wrong in this situation.  Further thanks goes to Jay Taylor for several helpful discussions on this matter, and who has allowed us to include his work on showing that Springer isomorphisms do not exist in this case.  Finally, we thank George McNinch for helpful clarifying discussions about several of his results that we have used throughout this paper.

\section{Preliminaries}

Throughout this paper $\Bbbk$ will denote an algebraically closed field of characteristic $p>0$.  We refer the reader to \cite[\S 1]{So} for any background material on Springer isomorphisms that is not covered here.

\subsection{Affine Group Schemes}

Let $R$ be a commutative ring., and $H$ an affine group scheme over $R$.  We denote by $R[H]$ the coordinate ring or affine algebra of $H$.  For each commutative $R$-algebra $S$, $H(S) = \text{Hom}_{R-alg}(R[H],S)$ is an abstract group.  Further, we can base change to an affine group scheme $H_S$ over $S$, having coordinate ring $S[H_S] \cong R[H] \otimes_R S$.  There is then an isomorphism of abstract groups $H(S) \cong H_S(S)$.  When referring to such a group we will use these notations interchangeably.

The Lie algebra of $H$ is $\mathfrak{h}$, and $\mathfrak{h}_S = \text{Lie}(H_S)$.  If $H$ is \textit{infinitesimally flat}, then there is an isomorphism $\mathfrak{h}_S \cong \mathfrak{h} \otimes_R S$ \cite[I.7.4(2)]{J}.

If $H$ is an affine group scheme over $\Bbbk$, we write $\mathcal{U}(H)$ to denote the unipotent variety of $H$, and $\mathcal{N}(\mathfrak{h})$ denotes the nilpotent variety of $\mathfrak{h}$.  Centralizers in $H$ will be denoted by $C_H(-)$, and centralizers in $\mathfrak{h}$ by $C_{\mathfrak{h}}(-)$.

\subsection{A Factorization Result}

Given affine schemes $\mathcal{X}$ and $\mathcal{Y}$ over $R$, a morphism $\varphi: \mathcal{X} \rightarrow \mathcal{Y}$ is given by maps $\varphi_A:\mathcal{X}(A) \rightarrow \mathcal{Y}(A)$, for each commutative $R$-algebra $A$, such that the assignment is functorial.  When no confusion may arise, we might simply denote $\varphi_A$ as $\varphi$.  The morphism $\varphi$ is equivalent to its comorphism $\varphi^*:R[\mathcal{Y}] \rightarrow R[\mathcal{X}]$.

We will repeatedly employ the following result about factorizations of affine scheme morphisms.  Though stated in more generality, the argument is essentially that given by McNinch in the proof of \cite[Proposition 7.2]{M}.

\begin{prop}\label{factorsthrough}
Let $\mathcal{X}$ be an affine scheme over $R$, and let $S$ be a commutative $R$-algebra for which the natural map $R[\mathcal{X}] \rightarrow R[\mathcal{X}] \otimes_R S$ is injective.  Let $\mathcal{Z}$ be an affine scheme over $R$, and $\mathcal{Y} \subseteq \mathcal{Z}$ a closed subscheme.  If $\phi:\mathcal{X} \rightarrow \mathcal{Z}$ is a morphism of schemes over $R$ such that $\phi_S: \mathcal{X}_S \rightarrow \mathcal{Z}_S$ factors through the inclusion of $\mathcal{Y}_S$ into $\mathcal{Z}_S$, then $\phi$ factors through the inclusion of $\mathcal{Y}$ into $\mathcal{Z}$.

In particular, these conditions are satisfied if $R=\Zp$, $S$ is a field of characteristic $0$, and $\mathcal{X}$ is an affine scheme over $\Zp$ such that $\Zp[\mathcal{X}]$ is a free $\Zp$-module.
\end{prop}

\begin{proof}
Let $J \subseteq R[\mathcal{Z}]$ denote the ideal defining the closed subgroup scheme $\mathcal{Y}$.  That is, $R[\mathcal{Y}] \cong R[\mathcal{Z}]/J$.  The homomorphism $\phi$ is equivalent to its comorphism
$$\phi^*: R[\mathcal{Z}] \rightarrow R[\mathcal{X}].$$
This morphism then factors through $\mathcal{Y}$ if and only if $\phi^*(J)=0$.  Now, on change of base to $S$, the morphism $\phi_S$ does factor through $\mathcal{Y}_S$, so we have that $\phi_S^*(J \otimes_R 1_S) = 0$.  On the other hand, $\phi_S^*(J \otimes_R 1_S) = \phi^*(J) \otimes_R 1_S \subseteq R[\mathcal{X}] \otimes_R S$.  By assumption, the map $R[\mathcal{X}] \rightarrow R[\mathcal{X}] \otimes_R 1_S$ is an isomorphism, thus $\phi^*(J) = 0$.
\end{proof}

\subsection{Chevalley Groups}\label{chevalley}

We refer to \cite[II.1.1]{J}, where one finds a list of references for the following facts.  There exists a smooth, infinitesimally flat, and split affine group scheme $G_{\mathbb{Z}}$ over $\mathbb{Z}$ that base changes to $G$.  In particular, $\mathbb{Z}[G_{\mathbb{Z}}]$ is a free $\mathbb{Z}$-module, finitely generated as a $\mathbb{Z}$-algebra, and $\gZ$ is a free $\mathbb{Z}$-module of finite rank.  If $G$ is of adjoint type, then the adjoint representation of $G$ on $\g$ arises over $\mathbb{Z}$, corresponding to a closed embedding of affine group schemes $G_{\mathbb{Z}} \rightarrow GL(\gZ)$.  We will generally work with these group scheme structures over $\Zp$, rather than all the way down to $\mathbb{Z}$.

We fix a split Borel subgroup scheme $B_{\mathbb{Z}} \le G_{\mathbb{Z}}$.  All standard parabolic subgroups over $\mathbb{Z}$ will contain $B_{\mathbb{Z}}$ as a subgroup scheme.  There is a semidirect decomposition $B_{\mathbb{Z}} \cong T_{\mathbb{Z}} \ltimes U_{\mathbb{Z}}$, with $T_{\mathbb{Z}}$ a split torus and $U_{\mathbb{Z}}$ a split unipotent subgroup scheme.

The root system of $G$ will be denoted as $\Phi$, the positive roots as $\Phi^+$, the simple roots as $\Pi$, and the Weyl group is $W$.  For each $\alpha \in \Phi$ there is a corresponding root subgroup $U_{\alpha,\mathbb{Z}}$.  Our choice of Borel subgroup $B_{\mathbb{Z}}$ corresponds to the negative root subgroups.

A prime $p$ is good for the root system of $G$ if it is not bad for the root system.  We have included in Table \ref{table:1} a list of all bad primes.  The Coxeter number of $G$ is denoted as $h$.

\subsection{Nilpotent Elements and Associated Cocharacters}

Let $\theta$ be a cocharacter of $G$ (that is, a morphism of algebraic groups $\theta: \mathbb{G}_m \rightarrow G$).  The adjoint action of $G$ on $\g$ restricts via $\theta$ to make $\g$ into a $\mathbb{G}_m$-module.  For each integer $j$, let $\g(\theta;j)$ denote the subspace of $\g$ having weight $j$.  It can be seen directly that $[\g(\theta;i),\g(\theta;j)] \subseteq \g(\theta;i+j)$. 

There is a parabolic subgroup $P(\theta) \le G$ determined by $\theta$ (cf. \cite[8.4.5]{Sp}).  The Lie algebra of $P(\theta)$ is precisely the subalgebra generated by the non-negatively graded elements.  That is, $\mathfrak{p}(\theta) =\sum_{j \ge 0} \g(\theta; j)$.  The Lie algebra of its unipotent radical $U(\theta)$ is the subalgebra $\sum_{j > 0} \g(\theta; j)$.

Let $X \in \Ng$.  Then $\theta$ is said to be associated to $X$ if $X \in \g(\theta;2)$, and if $\theta$ has image in the derived subgroup of $C_G(S)$, where $S$ is a maximal torus of $C_G(X)$.  Associated cocharacters exist in good characteristic, and can be viewed as a generalization of Jacobson-Morozov $\mathfrak{sl}_2$-triples.  Any two associated cocharacters for $X$ differ by conjugation by an element in $C_G(X)$.

If $\theta$ is associated to $X$, then $C_{\g}(X)$ (which in very good characteristic $= \text{Lie}(C_G(X))$) is a $\theta(\mathbb{G}_m)$-submodule of $\g$.  The weights of this submodule were originally worked out for $G_{\mathbb{C}}$ by Kostant in \cite{K}, and are twice the exponents of the Weyl group of $G_{\mathbb{C}}$.  More recently they were shown to be the same in very good characteristic by McNinch and Testerman \cite{MT}.

\begin{prop}\label{weightsofcenter}\cite[5.2.5]{MT}
Let $X \in \g$ be a regular nilpotent element, $\theta$ an associated cocharacter of $X$, and $r$ the rank of $G$.  Assume that $p$ is a very good prime for $G$.  Then $C_{\g}(X)$ has a basis of the form
$$\{X_1,X_2,\ldots,X_r\},$$
with $X=X_1$, and
$$X_i \in \g(\theta; 2k_i),$$
where $1=k_1 \le k_2 \le \cdots \le k_r$ are the exponents of the Weyl group of $G$.  Moreover, we can choose $X$ and $\theta$ over $\Zp$.  The subgroup scheme $C_{G_{\Zp}}(X)$ is smooth, and its Lie algebra is a free $\Zp$-module having a basis with respect to the action of $\theta(\mathbb{G}_{m,\Zp})$ as above.
\end{prop}

Let us now explain the final statement in the proposition a bit further.  The interested reader is referred to the proof in \cite{MT} for further details.

We can choose an element $X \in \gZp$ that is a regular nilpotent element in $\gC$, and such that $X \otimes_{\Zp} \Bbbk$ is regular nilpotent in $\g$.  The subgroup scheme $C_{G_{\Zp}}(X)$ is smooth, and on base change to $\Bbbk$ is $C_G(X \otimes_{\Zp} \Bbbk)$.  Further, $\text{Lie}(C_{G_{\Zp}}(X))$ is a free $\Zp$-module, and
$$C_{\g}(X \otimes_{\Zp} \Bbbk)= \text{Lie}(C_{G_{\Zp}}(X)) \otimes_{\Zp} \Bbbk.$$
One can then find a homomorphism of group schemes over $\Zp$,
$$\mathbb{G}_{m,\Zp} \rightarrow G_{\Zp},$$
which base changes over $\Bbbk$ to an associated cocharacter for $X \otimes_{\Zp} \Bbbk$.  Via this homomorphism, $\text{Lie}(C_{G_{\Zp}}(X))$ is a $\mathbb{G}_{m,\Zp}$-module.  One can therefore choose a basis over $\Zp$ for $\text{Lie}(C_{G_{\Zp}}(X))$ consisting of weight ``vectors'' for the diagonalizable action of $\mathbb{G}_{m,\Zp}$, with the weights given as in the proposition.

\begin{example}
Let $G=SL_n$, and let $X$ be the regular nilpotent element consisting of $1$'s on the superdiagonal and $0$'s elsewhere.  A particular basis for $C_{\g}(X)$ is
$$\{X,X^2,X^3,\ldots,X^{n-1}\}.$$
Since $X \in \g(\theta; 2)$, it follows that $X^i \in \g(\theta; 2i)$.  This agrees with the lemma, as the exponents of the Weyl group for $SL_n$ are: $1,2,\ldots,n-1$.  Furthermore, $X$ clearly comes from $\mathfrak{sl}_{n,\Zp}$.
\end{example}

We have included a listing of the exponents of the Weyl group for each root system in Table \ref{table:1}.  Observe that the exponents appear with multiplicity one in every case with one exception: if $n$ is even, then $n-1$ appears twice as an exponent for $D_n$.

\vspace{0.1in}

\begin{table}[b!]
\begin{tabular}{ | c | c | c | c |}
\hline
Root System & Exponents & Bad Primes & Coxeter Number\\
\hline
$A_n$ & $1,2,\ldots,n$ & none & $n+1$\\
\hline
$B_n$ & $1,3,5,\ldots,2n-1$ & $2$ & $2n$\\
\hline
$C_n$ & $1,3,5,\ldots,2n-1$ & $2$ & $2n$\\
\hline
$D_n$ & $1,3,5,\ldots,2n-3,$ and $n-1$ & $2$ & $2n-2$\\
\hline
$E_6$ & $1,4,5,7,8,11$ & $2,3$ & $12$\\
\hline
$E_7$ & $1,5,7,9,11,13,17$ & $2,3$ & $18$\\
\hline
$E_8$ & $1,7,11,13,17,19,23,29$ & $2,3,5$ & $30$\\
\hline
$F_4$ & $1,5,7,11$ & $2,3$ & $12$\\
\hline
$G_2$ & $1,5$ & $2,3$ & $6$\\
\hline
\end{tabular}
\bigskip
\caption{Numerical Data For Root Systems}
\label{table:1}
\end{table}

\vspace{0.1in}
We can use the decomposition of $C_{\g}(X)$ given above to construct maps for $G$ in a more explicit manner, similar to that for $SL_n$ (or $GL_n$).  In order to do so, we will require the following result.

\begin{lemma}\label{isregular}
Let $X \in \g$ be a regular nilpotent element.  Let $Y \in C_{\g}(X)$, and write
$$Y= aX + X^{\prime}, \qquad X^{\prime} \in \bigoplus_{i=2}^r C_{\g}(X)(\theta; 2k_i).$$
Then $Y$ is a regular nilpotent element if and only if $a \ne 0$.
\end{lemma}

\begin{proof}
We begin with the fact that since $X$ is a regular nilpotent element, $C_{\g}(X)$ is an abelian Lie algebra, consisting entirely of nilpotent elements.  If $Y \in C_{\g}(X)$, it follows that $C_{\g}(X) \subseteq C_{\g}(Y)$.  Now, $Y$ is regular if and only if this containment is an equality, which happens if and only if $C_{\g}(Y)\subseteq C_{\g}(X)$.  We will therefore show that this latter containment occurs if and only if $a \ne 0$.

First, suppose that $a \ne 0$.  After rescaling, we may assume that $a=1$.  Let $Z \in C_{\g}(Y)$, and write
$$Z = \sum_j Z_j, \quad Z_j \in \g(\theta;j).$$
Since $[Y,Z]=0$, and $Y=X + X^{\prime}$, we have that
$$[X,\sum_j Z_j]=-[X^{\prime},\sum_j Z_j].$$
Choose $j_0$ minimal such that $Z_{j_0} \ne 0$.  We then have
$$[X,Z_{j_0}]=-[X,\sum_{j>j_0} Z_j] - [X^{\prime},\sum_{j \ge j_0} Z_j].$$
As $X^{\prime}$ is the sum of weight vectors for $\theta(\mathbb{G}_m)$ having weight strictly greater than $2$, it follows that the right side of the equation is a sum of weight vectors of weight $>j_0+2$, forcing $[X,Z_{j_0}]=0$.  But then $Z_{j_0} \in C_{\g}(X) \subseteq C_{\g}(Y)$, and so $Z-Z_{j_0} \in C_{\g}(Y)$.  We may now replace $Z$ with $Z-Z_{j_0}$ and repeat the argument.  Continuing in this way, we find that each $Z_j$ commutes with $X$, so that $Z$ does also.  Therefore, $Z \in C_{\g}(X)$, and so $C_{\g}(Y) \subseteq C_{\g}(X)$.

On the other hand, suppose that $a=0$.  The parabolic subgroup $P(\theta)$ corresponding to $\theta$ is a Borel subgroup of $G$ since $X$ is regular.  We have here that $Y \in \sum_{j > 2} \g(\theta; j)$.  Since this subalgebra does not contain $X$, it is a proper subalgebra of $\text{Lie}(U(\theta))$.  It is stable under the adjoint action of $P(\theta)$ (recall that the $\theta(\mathbb{G}_m)$-weights in $\mathfrak{p}(\theta)$ are non-negative).  It follows that $Y$ cannot be a regular nilpotent element as its $P(\theta)$-orbit will not coincide with the $P(\theta)$-orbit of $X$, thus $Y$ is not a Richardson element of the Borel subgroup $P(\theta)$).
\end{proof}

\section{Constructing All Springer Isomorphisms for $G$}\label{constructions}

\subsection{}

Our goal is to provide increasingly detailed and useful descriptions of the variety of all Springer isomorphisms for $G$, with $SL_n$ serving as our model.  We recall the observations of Serre in the appendix to \cite{M2}.  Given a regular nilpotent element $X \in \g$, a Springer isomorphism is defined completely by where it sends $X$, with any regular unipotent element in $C_G(X)$ a valid destination.  With $N_G(C_G(X))$ denoting the normalizer in $G$ of $C_G(X)$, there is an identification between the regular unipotent elements in $C_G(X)$ and the cosets of $C_G(X)$ in $N_G(C_G(X))$.  Thus, the set of all Springer isomorphisms for $G$ is parameterized by the algebraic group $N_G(C_G(X))/C_G(X)$.

We will find it necessary to pass back and forth between the first two equivalent statements in the next lemma, both for $G$ and for $G_{\mathbb{C}}$.  We add a third equivalence to give a more complete picture.

\begin{lemma}\label{springeroversubgroups}
Let $\mathbb{F}$ be any algebraically closed field having separably good characteristic for $G$ (including characteristic $0$), and let $X$ be a regular nilpotent element of $\gF$.  Then the following sets are each in bijection with $N_{G_{\mathbb{F}}}(C_{G_{\mathbb{F}}}(X))/C_{G_{\mathbb{F}}}(X)$:
\begin{itemize}
\item[(a)] $G_{\mathbb{F}}$-equivariant isomorphisms from $\mathcal{N}(\gF)$ to $\mathcal{U}(G_{\mathbb{F}})$.
\item[(b)] $B_{\mathbb{F}}$-equivariant isomorphisms from $\mathfrak{u}_{\mathbb{F}}$ to $U_{\mathbb{F}}$. 
\item[(c)] $N_{G_{\mathbb{F}}}(C_{G_{\mathbb{F}}}(X))$-equivariant isomorphisms from $C_{\gF}(X)$ to $C_{G_{\mathbb{F}}}(X)^0$.
\end{itemize}
\end{lemma}

\begin{proof}
If we intersect the regular nilpotent orbit in $\mathcal{N}(\gF)$ with $\mathfrak{u}_{\mathbb{F}}$ it results in a single, necessarily dense, $B_{\mathbb{F}}$-orbit (the Richardson orbit for $B_{\mathbb{F}}$).  The analogous statement holds for $C_{\gF}(X)$ and $N_{G_{\mathbb{F}}}(C_{G_{\mathbb{F}}}(X))$.  Thus, in each setting an equivariant isomorphism corresponds to sending $X$ to a regular unipotent element having the same centralizer inside the given subgroup of $G_{\mathbb{F}}$.  The equality of these isomorphisms then follows from the fact that
$$C_{G_{\mathbb{F}}}(X)=C_{B_{\mathbb{F}}}(X)=C_{N_{G_{\mathbb{F}}}(C_{G_{\mathbb{F}}}(X))}(X).$$
\end{proof}

We now return our attention to $G$ (though what follows also holds for $G_{\mathbb{F}}$ with $\mathbb{F}$ as above).  In \cite{MT}, McNinch and Testerman verified Serre's expectation that as an algebraic group $N_G(C_G(X))/C_G(X)$ is isomorphic to $\mathbb{G}_m \ltimes V$, with $V$ a connected unipotent group of dimension $r-1$ ($r$ being the rank of $G$).  Thus, the ring of functions on the variety of all Springer isomorphisms for $G$ is isomorphic to
$$\Bbbk[x_1,{x_1}^{-1},x_2,x_3,\ldots,x_r].$$

For $SL_n$, the $\Bbbk$-points of this variety can be seen directly.  Every Springer isomorphism is given by the map sending $X$ to $I_n + a_1X + \cdots a_{n-1}X^{n-1}$, $a_1 \in \Bbbk^{\times}$.  Moreover, it is evident in this case that this variety of Springer isomorphisms in fact arises as a scheme over $\mathbb{Z}$. 

This general recipe for a Springer isomorphism of $SL_n$ can be split into the composition of two morphisms.  The first is an $SL_n$-equivariant variety automorphism of $\mathcal{N}(\mathfrak{sl}_n)$ defined by the rule
$$X \mapsto a_1X + \cdots + a_{n-1}X^{n-1},$$
followed by the Springer isomorphism for $SL_n$ that sends $X$ to $I_n + X$.  Note that Lemma \ref{isregular} verifies that requiring $a_1 \in \Bbbk^{\times}$ is both necessary and sufficient for the first morphism to be an isomorphism.  This, then, shows us how to make a first level generalization of the construction for $SL_n$.

\begin{prop}\label{firstdescription}
Let $r$ be the rank of $G$.  Fix a regular nilpotent element $X \in \g$, an associated cocharacter $\theta$ of $X$, and a basis of weight vectors $C_{\g}(X)$ as in Lemma \ref{weightsofcenter}.  Given any particular Springer isomorphism $\phi_0$ for $G$, the variety of Springer isomorphisms for $G$ identifies with the variety
$$\{(a_1,a_2,\ldots,a_r), \mid a_i \in \Bbbk, a_1 \in \Bbbk^{\times}\},$$
by sending the $r$-tuple $(a_1,a_2,\ldots,a_{r})$ to the unique Springer isomorphism $\phi$ with
$$\phi(X) = \phi_0(a_1X + a_2X_2 + \cdots + a_rX_r).$$
\end{prop}

\begin{proof}
We only need to show that $\phi_0(a_1X + a_2X_2 + \cdots + a_rX_r)$ is a regular unipotent element in $C_G(X)$.  But this follows from Lemma \ref{isregular} and the fact that $\phi_0$ is a Springer isomorphism.
\end{proof}

This description can be parsed out into finer detail.  Continuing with the notation above, consider the mapping
$$\gamma_i:G.X \rightarrow G.X_i$$
given by sending $g.X$ to $g.X_i$, for all $g \in G$.  This is well defined because $C_G(X) \le C_G(X_i)$.  Moreover, it is a morphism of varieties, as our assumption on the characteristic guarantees that each nilpotent orbit is isomorphic to the quotient of $G$ by the corresponding stabilizer subgroup (work here with $GL_n$ rather than $SL_n$ if $p \mid n$).  Thus, $\gamma_i$ is given from the sequence of morphisms
$$G.X \xrightarrow{\sim} G/C_G(X) \rightarrow G/C_G(X_i) \xrightarrow{\sim} G.X_i.$$
By normality, $\gamma_i$ extends uniquely to a $G$-equivariant morphism from $\Ng$ to $\Ng$ that we will also denote as $\gamma_i$.  In this notation, the statement in the previous proposition could be expressed as saying that $\phi$ is the Springer isomorphism such that for every $Y \in \Ng$, $$\phi(Y)=\phi_0(a_1\gamma_1(Y) + \cdots + a_r\gamma_r(Y)).$$
The following analysis will come in handy later, and shows that each $\gamma_i$ behaves nicely with respect to cocharacters of $G$. 

\begin{lemma}\label{likemultilinear}
Let $k_i$ be the $i$-th exponent of the Weyl group of $G$, and let $\tau$ be a cocharacter of $G$.  If $Y \in \g(\tau;z)$, then $\gamma_i(Y) \in \g(\tau;z k_i)$.
\end{lemma}

\begin{proof}
Let $a \in \mathbb{G}_m$.  By $G$-equivariance, we have that
$$\gamma_i(a^2X)=a^{2k_i}\gamma_i(X_i).$$
Replacing $a$ with $\sqrt{a}$, it immediately follows from $G$-equivariance that for all $g \in G$,
$$\gamma_i(a(g.X))=a^{k_i}\gamma_i(g.X).$$
This means that the following diagram of $G$-equivariant morphisms
\[
  \begin{tikzcd}
    \Ng \arrow{r}{a} \arrow{d}{\gamma_i} & \Ng \arrow{d}{\gamma_i}\\
    \Ng \arrow{r}{a^{k_i}} & \Ng,
  \end{tikzcd}
\]
where the horizontal maps are scalar multiplication by the indicated elements, commutes on the regular orbit, hence by density on all of $\Ng$.  In summary, for every $Y \in \Ng$ and $a \in \mathbb{G}_m$ we have an equality
\begin{equation}\label{rescaling}
\gamma_i(aY)=a^{k_i}\gamma_i(Y).
\end{equation}
In particular, we have $\gamma_i(\Bbbk Y)=\Bbbk (\gamma_i(Y))$.  Suppose now that $Y \in \g(\tau;z)$ for some $z\in \mathbb{Z}$.  Comparing the actions of $\tau(\mathbb{G}_m)$ on both sides, since $\tau(\mathbb{G}_m).\Bbbk Y = \Bbbk Y$, it follows that $\gamma_i(Y) \in \g(\tau;d)$ for some $d$.

By $\tau(\mathbb{G}_m)$-equivariance we have for each $s \in \mathbb{G}_m$,
$$\gamma_i(s^zY)=s^d\gamma_i(Y).$$
But by (\ref{rescaling}), we have
$$\gamma_i(s^zY)={(s^z)}^{k_i}\gamma_i(Y).$$
Therefore $d=zk_i$.
\end{proof}

\begin{cor}
Let $j \ge 0$ be an integer, and $\tau$ a cocharacter of $G$.  If $\mu$ is any $G$-equivariant variety automorphism of $\Ng$, then
$$\mu\left(\sum_{i \ge j} \g(\tau;i)\right) \subseteq \sum_{i \ge j} \g(\tau;i).$$ 
\end{cor}

\begin{proof}
By the observations above, there are elements $a_i \in \Bbbk$, $a_1 \ne 0$, such that $\mu = a_1\gamma_1 + \cdots + a_r\gamma_r$.  The result then follows by the previous lemma which shows that each $\gamma_i$ has this property.
\end{proof}

\begin{remark}\label{GGGR}
We expect that with a bit of work, this result can be used to verify the suggestion by Taylor in \cite[Remark 4.7]{Ta} that ``most'' Springer isomorphisms for $G$ restrict to `Kawanaka isomorphisms' on the unipotent radicals of parabolic subgroups of $G$ (though, ``most'' must be restricted to those arising over $\mathbb{F}_p$ as there is a condition on compatibility with the Frobenius morphism).  See \textit{loc. cit.} for more on Kawanaka isomorphisms, and why it is important in the study of Generalized Gelfand-Graev representations to find one coming from a Springer isomorphism for $G$. 
\end{remark}

\subsection{Nondegenerate trace forms arising over $\Zp$}
Lemma \ref{likemultilinear} indicates that the morphism $\gamma_i$ behaves somewhat like the multiplicative power morphism
$$\mathcal{N}(\mathfrak{sl}_n) \rightarrow \mathcal{N}(\mathfrak{sl}_n)$$
that sends $X$ to $X^i$ (where, for $SL_n$, $i$ is the $i$-th exponent of the Weyl group).  We will now show that we can construct, for certain $i$, ``multiplicative power'' morphisms $m^i$ that are independent of a choice of basis for $C_{\g}(X)$ and also arise over $\Zp$.  We are able to define these maps by working with a fixed well-known representation of $G$ (or a group isogenous to $G$).  For a group of classical type this will be the natural representation, and for groups of exceptional type we will use the adjoint representation.  These representations were utilized previously by Bardsley and Richardson to give another proof of the existence of Springer isomorphisms \cite{BR}.

The key feature of these representations is that the trace form on the ambient matrix algebra restricts to a non-degenerate symmetric bilinear form on the embedded Lie algebra, splitting the matrix algebra into a direct sum of the Lie subalgebra and its orthogonal complement.  In order to move back and forth between characteristics, it is a crucial point that the non-degeneracy of the trace form mod $p$ leads to a splitting of the matrix algebra at the level of $\Zp$-modules.  We will see that this leads to a particularly nice lifting of the $[p]$-mapping in characteristic $p$.

Let $G_{\Zp}$ be a split $\Zp$-form of $G$ as in Section \ref{chevalley}.

\begin{prop}\label{splitoverZp}
Suppose there is a closed embedding of affine group schemes over $\Zp$,
$$\varphi:G_{\Zp} \rightarrow GL_{n,\Zp},$$
with the property that over $\mathbb{F}_p$, the trace form on $\mathfrak{gl}_{n,\mathbb{F}_p}$ is non-degenerate when restricted to $d\varphi(\gFp)$.  Then the adjoint action of $G_{\Zp}$ on $\mathfrak{gl}_{n,\Zp}$ splits into a direct sum of $G_{\Zp}$-modules
$$\mathfrak{gl}_{n,\Zp} \cong d\varphi(\gZp) \oplus \mathfrak{m}_{\Zp},$$
where $\mathfrak{m}_{\Zp}$ is the orthogonal complement to $d\varphi(\gZp)$ under the trace form and contains the identity $n \times n$-matrix $I_n$.
\end{prop}

\begin{proof}
The trace form on $\mathfrak{gl}_{n,\Zp}$ restricts to a symmetric bilinear form on the free $\Zp$-module $d\varphi(\gZp)$.  Because this form is nondegenerate over $\Fp$, it is a basic fact about local rings that the induced map
$$\gZp \rightarrow \text{Hom}_{\Zp}(\gZp,\Zp), \quad X \mapsto (d\varphi(X),d\varphi(\underline{\quad})),$$
is an isomorphism of $\Zp$-modules (cf. \cite{B}[Proposition 1.4]).  Thus, the elements in $\mathfrak{gl}_{n,\Zp}$ that are orthogonal to $d\varphi(\gZp)$ under this form constitute a $\Zp$-module complement.  Denote this complement as $\mathfrak{m}_{\Zp}$.

The adjoint representation of $GL_{n,\Zp}$, composed with $\varphi$,  defines a homomorphism of group schemes over $\Zp$,
$$\tilde{\varphi}:G_{\Zp} \rightarrow GL(\mathfrak{gl}_{n,\Zp}).$$
There is also a closed embedding of affine group schemes over $\Zp$,
$$GL(d\varphi(\gZp)) \times GL(\mathfrak{m}_{\Zp}) \rightarrow GL(\mathfrak{gl}_{n,\Zp}).$$
Over $\mathbb{C}$, it is easy to see that the orthogonal complement to $d\varphi(\gC)$ under the trace form on $\mathfrak{gl}_{n,\mathbb{C}}$ is the subspace $\mathfrak{m}_{\mathbb{C}} = \mathfrak{m}_{\Zp} \otimes_{\Zp} \mathbb{C}$.  It then follows that over $\mathbb{C}$ the morphism $\tilde{\varphi}_{\mathbb{C}}$ factors as
$$G_{\mathbb{C}} \rightarrow GL(d\varphi(\gC)) \times GL(\mathfrak{m}_{\mathbb{C}}) \hookrightarrow GL(\mathfrak{gl}_{n,\mathbb{C}}).$$
Because $\Zp[G_{\Zp}]$ is a free module over $\Zp$, by Proposition \ref{factorsthrough} the morphism over $\Zp$ also factors as
$$G_{\Zp} \rightarrow GL(d\varphi(\gZp)) \times GL(\mathfrak{m}_{\Zp})  \hookrightarrow GL(\mathfrak{gl}_{n,\Zp}),$$
proving that $\mathfrak{gl}_{n,\Zp}$ factors as a $G_{\Zp}$-module.  Finally, since $d\varphi(\gZp)$ consists of trace $0$ matrices, the identity matrix $I_n$ is orthogonal to $d\varphi(\gZp)$ under the trace form. 
\end{proof}

This last result says that for any commutative $\Zp$-algebra $R$, there is a decomposition
$$\mathfrak{gl}_{n,R} \cong d\varphi_R(\gR) \oplus \mathfrak{m}_{R},$$
which can equivalently be expressed as
$$\mathfrak{gl}_{n,\Zp} \otimes_{\Zp} R \cong (d\varphi(\gZp) \otimes_{\Zp} R) \oplus (\mathfrak{m}_{\Zp} \otimes_{\Zp} R).$$
We will often refer to the two projection morphisms simply as $\pi_{\mathfrak{g}}$ and $\pi_{\mathfrak{m}}$, without reference to the particular ring we are working over.  Observe that for any $X \in d\varphi(\gZp)$, we have
\begin{equation}\label{projectissame}
\pi_{\mathfrak{g}}(X \otimes_{\Zp} 1_R) = \pi_{\mathfrak{g}}(X) \otimes_{\Zp} 1_R.
\end{equation}

We now recall the following result due to Bardsley and Richardson.

\begin{prop}\label{BardselyRichardson}\cite[\S 9.3]{BR}
Let $\varphi$ be as in the previous proposition, and suppose that $\mathbb{F}$ is a field over $\Zp$.  Then the composite of morphisms
$$G_{\mathbb{F}} \xrightarrow{\varphi} GL_{n,\mathbb{F}} \xrightarrow{\pi_{\mathfrak{g}}} d\varphi(\gF)$$
restricts to a $G_{\mathbb{F}}$-equivariant isomorphism from $\mathcal{U}(G_{\mathbb{F}})$ to $\mathcal{N}(\gF)$.
\end{prop}

In Table \ref{table:2} we have fixed, for each root system, a representation over $\Zp$ satisfying the conditions in Proposition \ref{splitoverZp}.  It is well known in these cases that the representations each have non-degenerate trace forms mod $p$ for the primes listed.

\begin{table}[h!]
\centering
\begin{tabular}{ | c | c | c | c |}
\hline
Root System & Condition on prime & Isogeny Type & Representation\\ 
\hline
$A_n$ & $p \nmid (n+1)$ & $SL_{n+1}$ & Natural\\
\hline
$B_n$ & $p \ne 2$ & $SO_{2n+1}$ & Natural\\
\hline
$C_n$ & $p \ne 2$ & $Sp_{2n}$ & Natural\\
\hline
$D_n$ & $p \ne 2$ & $SO_{2n}$ & Natural\\
\hline
$E_6,E_7,E_8,F_4,G_2$ & $p$ good & Adjoint type & Adjoint\\
\hline
\end{tabular}
\bigskip
\caption{Representations Over $\Zp$ With Nondegenerate Trace Forms}
\label{table:2}
\end{table}

\vspace{0.1in}
\begin{remark}
Table \ref{table:2} does not cover $SL_n$ in all separably good characteristics, only in very good characteristic (such an embedding does not exist in non-very good characteristic).  However, the isomorphism given by Bardsley and Richardson for $SL_n$ in very good characteristic is the map sending a unipotent element $u$ to $u-I_n$, which of course is also defined in the not-very-good case, it just does not come about via a projection from $\mathfrak{gl}_n$ onto $\mathfrak{sl}_n$ in that situation.
\end{remark}

\begin{lemma}\label{projectionisinvertible}
The scheme morphism $\pi_{\mathfrak{g}}: G_{\mathbb{Z}_p} \rightarrow \gZp$ restricts to a $B_{\Zp}$-equivariant isomorphism of schemes $U_{\Zp} \xrightarrow{\sim} \mathfrak{u}_{\Zp}$.
\end{lemma}

\begin{proof}
The projection morphism restricts to a scheme morphism from $U_{\Zp}$ to $\gZp$.  This projection, over $\mathbb{C}$, restricts to an isomorphism from $U_{\mathbb{C}}$ to $\mathfrak{u}_{\mathbb{C}}$, as follows from the argument given in Lemma \ref{springeroversubgroups} and the fact that this projection restricts to the (inverse of) a Springer isomorphism for $G_{\mathbb{C}}$ by Proposition \ref{BardselyRichardson}.  Since $\Zp[U_{\Zp}]$ is a free $\Zp$-module (indeed, a polynomial ring over $\Zp$), the fact that $\pi_{\mathfrak{g}}$ takes $U_{\mathbb{C}}$ to $\mathfrak{u}_{\mathbb{C}}$ implies by Proposition \ref{factorsthrough} that the projection morphism restricts of a morphism of $\Zp$-schemes $\phi: U_{\Zp} \rightarrow \mathfrak{u}_{\Zp}$.  It is evidently $B_{\Zp}$-equivariant.

Proposition \ref{BardselyRichardson} assures that not only is $\phi_{\mathbb{C}}$ an isomorphism, but that $\phi_{\Bbbk}$ is also.  We have that $\Zp[U_{\Zp}]$ and $\Zp[\mathfrak{u}_{\Zp}]$ are polynomial rings over $\Zp$, graded by $T_{\Zp}$.  Since $B_{\Zp}$ consists of negative roots, we can index polynomial generators by the positive roots, so that
$$\Zp[U_{\Zp}] \cong \Zp[\{t_{\alpha}, \alpha \in \Phi^+\}]$$
and
$$\Zp[\mathfrak{u}_{\Zp}] \cong \Zp[\{u_{\alpha}, \alpha \in \Phi^+\}].$$
Now, the $T_{\Zp}$-equivariance means that $\phi^*$ must send $u_{\alpha}$ to a sum of elements of weight $\alpha$.  We know that $\phi^*$ is injective because $\phi^*_{\mathbb{C}}$ is injective.  To see that it is also surjective, we argue by induction on the height of a root $\alpha$ that $t_{\alpha}$ is in the image of $\phi^*$.  Indeed, if $\alpha$ is a simple root, then we have that $\phi^*(u_{\alpha})=ct_{\alpha}$ for some $c \in \Zp$.  Since $\phi^*_{\Bbbk}$ is injective, $c$ must be a unit in $\Zp$, hence $\phi^*(c^{-1}u_{\alpha})=t_{\alpha}$.  Suppose now that $\alpha$ has arbitrary height, and that $t_{\beta}$ is in the image of $\phi^*$ for all $\beta$ having height less than $\alpha$.  Then we have that
$$\phi^*(u_{\alpha})=bt_{\alpha} + \text{higher degree monomials},$$ 
for some $b \in \Zp$, that we can again argue must be a unit in $\Zp$.  By the induction hypothesis, the ``higher degree monomials'' (consisting of products of the $t_{\beta_i}$ for a sum of roots $\beta_i$ adding to $\alpha$) are in the image of $\phi^*$.  Thus $t_{\alpha}$ will also be in the image of $\phi^*$.  So, $\phi^*$ is bijective, hence $\phi$ is an isomorphism.
\end{proof}

\subsection{Multilinear maps and powers of elements in $\gZp$}

We now show that the representations listed in Table \ref{table:2} can be used to obtain multilinear maps from $(\gZp)^{\times i}$ to $\gZp$ (different from those coming  from repeated applications of the Lie bracket on $\gZp$).  These maps do not behave completely like the multilinear maps coming from the structure of an associative algebra (for example, they are not associative), but have enough features to do what we need.  In particular, they allow us to define multiplicative power maps on $\gZp$.

\begin{defn}
Let $G_{\Zp}$ and $\varphi$ be as in Table \ref{table:2}.  For each $i > 0$, we have a $\Zp$-multilinear morphism
$$m_i:(\gZp)^{\times i} \rightarrow \gZp$$
from the composition of morphisms
$$(\gZp)^{\times i} \xrightarrow{(d\varphi)^{\times i}} (\mathfrak{gl}_{n,\Zp})^{\times i} \xrightarrow{mult.} \mathfrak{gl}_{n,\Zp} \xrightarrow{(d\varphi)^{-1} \circ \pi_{\mathfrak{g}}} \gZp.$$
This allows us to define a morphism of schemes over $\Zp$
$$m^i:\gZp \rightarrow \gZp$$
defined by
$$m^i(X)=m_i(\underbrace{X,X,\ldots,X}_{i \textup{ times}}).$$
\end{defn}

\begin{prop}\label{allbasechangeswork}
Let $R$ be a commutative $\Zp$-algebra.  The following properties hold:
\begin{itemize}
\item[(a)] $m_i$ and $m^i$ are $G_{R}$-equivariant morphisms for every $i$.
\item[(b)] $[X,Y] = m_2(X,Y)-m_2(Y,X)$, for all $X,Y \in \gR$.
\item[(c)] $m^i(X) \in C_{\gR}(X)$ for all $X \in \gR$.
\item[(d)] For every $r \in R$ and $X \in \gR$, we have $m^i(rX) = r^i(m^i(X))$.  In particular, if $\tau$ is a cocharacter of $G$, and $X \in \g(\tau;z)$, then $m^i(X) \in \g(\tau;z i)$.
\item[(e)] $m^p$, as a morphism on $\g$, is just the $[p]$-th power map.  
\end{itemize}
\end{prop}

\begin{proof}
(a) This is just a consequence of the fact that both $\pi_{\mathfrak{g}}$ and multiplication in $\mathfrak{gl}_{n,\Zp}$ are $G_{\Zp}$-equivariant morphisms of schemes ($G_{\Zp}$ acting by conjugation via $\varphi$).

(b) This follows from the fact that
\begin{align*}
d\varphi([X,Y])= &\pi_{\mathfrak{g}}(d\varphi([X,Y]))\\
= & \pi_{\mathfrak{g}}(d\varphi(X)d\varphi(Y)-d\varphi(Y)d\varphi(X))\\
= & \pi_{\mathfrak{g}}(d\varphi(X)d\varphi(Y))- \pi_{\mathfrak{g}}(d\varphi(Y)d\varphi(X))).\\
\end{align*}

(c) The splitting
$$\mathfrak{gl}(\gZp) = d\varphi(\gZp) \oplus \mathfrak{m}_{\Zp}$$
is stable under the commutator action of $d\varphi(\gZp)$.  Since
$$[d\varphi(X),d\varphi(X)^i]=0,$$
it must also be the case that
$$[d\varphi(X),\pi_{\mathfrak{g}}(d\varphi(X)^i)]=0=[d\varphi(X),\pi_{\mathfrak{m}}(d\varphi(X)^i)].$$
Therefore, $[X,m^i(X)]=0$.

(d) The first statement follows from the fact that $d\varphi(rX)^i=r^i(d\varphi(X)^i)$, and that every other morphism in the definition of $m^i$ is $R$-linear (when considered over $R$).  The second fact follows from an argument as in the proof of Lemma \ref{likemultilinear}.

(e) Suppose now that $X \in \g$.  Since $d\varphi_{\Bbbk}$ is an embedding of $[p]$-restricted Lie algebras, we have
$$\pi_{\g}(d\varphi_{\Bbbk}(X)^p) = \pi_{\g}(d\varphi_{\Bbbk}(X^{[p]})) = d\varphi_{\Bbbk}(X^{[p]}).$$
From this it is clear that $m^p(X)=X^{[p]}$, verifying the claim.
\end{proof}

\begin{remark}\label{Delignelifting} The significance of lifting the $[p]$-mapping to $\Zp$ was first pointed out to us by Deligne, who had observed another way to do it via Springer's explicit computations in \cite{Sp2}.
\end{remark}

\section{Generalized Exponential Maps In Characteristic $p$}

Let $G$ be a simple algebraic group over $\Bbbk$ in separably good characteristic.

\subsection{Definition Of Generalized Exponential Maps}

We begin with a definition that is suggested by the statement of \cite[Theorem 4.1]{So}, though never explicitly stated as such.  Recall that a nilpotent element $X$ is said to have nilpotent order $p^m$ if $m$ is the smallest integer such that $X^{[p^m]} = 0$.  The group $\mathcal{W}_m$ is the group of Witt vectors over $\Bbbk$ of length $m$.

\begin{defn}\label{embeddingproperty}
A Springer isomorphism $\phi: \Ng \rightarrow \mathcal{U}(G)$ is called a generalized exponential map if for each integer $m$, and for each $X \in \Ng$ of nilpotence degree $p^m$, there is a closed embedding of algebraic groups
$$\phi_X: \mathcal{W}_m \rightarrow G$$
given by the map
$$(a_0,a_1,\ldots,a_{m-1}) \mapsto \phi(a_0X)\phi(a_1X^{[p]})\cdots\phi(a_{m-1}X^{[p^{m-1}]}).$$
\end{defn}

Note that a consequence of this definition is that if $\phi$ is such a Springer map, then it necessarily satisfies the relation
$$\phi(X^{[p]})=\phi(X)^p$$
for all $X$.  This adds an extra level of rigidity when considering issues related to the uniqueness of such maps (see Section \ref{uniqueness}).

We recall now the main theorem of \cite{So}, stated in terms of this definition.

\begin{theorem}\cite[Theorem 4.1]{So}
Generalized exponential maps exist in separably good characteristic.
\end{theorem}

\subsection{The Artin-Hasse Exponential For Classical Groups}\label{Artinforclassical}

The definition above (and the work in \cite{So}) was motivated by the following well-known fact for $GL_n$.  First, recall that the Artin-Hasse exponential series for $p$ is the formal power series
$$E_p(t) = \exp\left(t + \frac{t^p}{p} + \frac{t^{p^2}}{p^2} + \cdots\right).$$
The coefficients of this series lie in $\mathbb{Z}_{(p)}$ \cite{D}.

Suppose that $Y \in \mathfrak{gl}_{n,\Zp}$ is a nilpotent matrix.  There is some $m \ge 0$ that is minimal with respect to the property that $Y^{p^m}=0$.  We then get a closed embedding of group schemes over $\Zp$,
$$f_Y: \mathcal{W}_{m,\Zp} \rightarrow GL_{n,\Zp},$$
where for every commutative $\Zp$-algebra $A$ and every $(a_0,\ldots,a_{m-1}) \in \mathcal{W}_{m,A}$,
$$f_Y((a_0,\ldots,a_{m-1})) = E_p(a_0Y)\cdots E_p(a_{m-1}Y^{p^{m-1}}).$$

When working over $\mathbb{C}$ (or $\mathbb{Q}$ even), $E_p(Y)$ can be expressed as
\begin{equation}\label{Epfactor}
E_p(Y)=\exp(Y)\exp\left(\frac{Y^p}{p}\right)\cdots\exp\left(\frac{Y^{p^{\ell-1}}}{p^{\ell-1}}\right).
\end{equation}
Suppose now that we have a closed embedding of affine groups schemes $\varphi: G_{\Zp} \rightarrow GL_{n,\Zp}$.  Consider an element $X \in \gZp$ such that $X \otimes_{\Zp} 1_{\Bbbk}$ and $X \otimes_{\Zp} 1_{\mathbb{C}}$ are regular nilpotent elements in $\mathfrak{g}$ and $\mathfrak{g}_{\mathbb{C}}$ respectively (the latter element we will denote simply as $X$).  Let $m$ again be the minimal integer such that $d\varphi(X)^{p^m}=0$.  If, for each $i<m$, it holds that
$$d\varphi(X)^{p^i} \in d\varphi(\gC),$$
then by (\ref{Epfactor}) we have $E_p(d\varphi(X)) \in \varphi(G_{\mathbb{C}})$.  We then have that the morphism
$$f_{d\varphi(X)}: \mathcal{W}_{m,\mathbb{C}} \rightarrow GL_{n,\mathbb{C}}$$
factors through $\varphi(G_{\mathbb{C}})$, so that by Proposition \ref{factorsthrough} the morphism over $\Zp$ factors through $\varphi(G_{\Zp})$ (this is more or less the argument by McNinch in \cite[\S 7.5]{M}).

Over $\Bbbk$, $E_p$ also defines a Springer isomorphism for $GL_n$.  Since $d\varphi(X \otimes_{\Zp} 1_{\Bbbk})$ is carried into $\varphi(G)$, then by $\varphi(G)$-equivariance, and the density of the regular orbit in $d\varphi(\mathcal{N}(\g))$, it follows that this Springer isomorphism for $GL_n$ restricts to one for $\varphi(G)$.  It is, by definition, a generalized exponential map.

Now, if $G$ is a classical group in its natural representation, and $p \ne 2$ for types $B,C$, or $D$, then we have for every $X \in \gZp$ that $d\varphi(X)^{p^i}=d\varphi(m^{p^i}(X)) \in d\varphi(\gZp)$, so by the reasoning above we obtain a generalized exponential map for $G$ (this was covered in \cite{So}).

\subsection{A Modified Artin-Hasse Exponential For Types $E-G$}\label{modified}

We now adjust the argument in the previous subsection to obtain generalized exponential maps for groups having a root system of exceptional type.  This very straightforward adaptation is due to Deligne, who shared it in a private communication on which we were copied.  However, we give an alternate proof that the morphism is defined over $\Zp$, utilizing the work in Section \ref{constructions}.

The exponential map defines a canonical Springer isomorphism for $G_{\mathbb{C}}$.  For each $i$, let $m^{p^i}$ denote the morphism $(m^p)^i$.  If $X$ is a regular nilpotent element of $\gC$, and $\theta$ an associated cocharacter, then by properties (c) and (d) of Proposition \ref{allbasechangeswork}, it follows that $m^{p^i}(X)$ is an element in $C_{\gC}(X)$ that is in $\gC(\theta;2p^i)$.  If $p^i \ge h$, then this element is $0$.  Thus, there is some $i$ such that $m^{p^j}(X)=0$ for all $j \ge i$, and all $X \in \mathcal{N}(\gC)$.  With $\phi_0=\exp$, we obtain by Proposition \ref{firstdescription} a Springer isomorphism $\tilde{E}_p$ for $G_{\mathbb{C}}$ where
\begin{equation}\label{AHequation}
\tilde{E}_p(X)=\exp\left(X + \frac{m^p(X)}{p} + \frac{m^{p^2}(X)}{p^2} + \cdots\right).
\end{equation}
In the case of the groups of types $A-D$, $\tilde{E}_p$ is precisely the Springer isomorphism given by evaluating the Artin-Hasse exponential series on nilpotent matrices.

Now, in in all types ($p$ separably good) $\tilde{E}_p$ defines (and is equivalent to) a $B_{\mathbb{C}}$-equivariant isomorphism from $\mathfrak{u}_{\mathbb{C}}$ to $U_{\mathbb{C}}$, which we will also denote as $\tilde{E}_p$.  We will show that this morphism is defined over $\Zp$, and gives rise to a generalized exponential map in characteristic $p$.  First, we require the following technical lemma.

\begin{lemma}\label{Emeansdefined}
Let $G_{\Zp}$ and $\varphi$ be one of the groups and representations indicated in Table \ref{table:2}.  Then the morphism of schemes over $\mathbb{C}$,
$$\sigma: \gC \rightarrow \mathfrak{m}_{\mathbb{C}},$$
where 
$$\sigma(X)=\frac{1}{p}\pi_{\mathfrak{m}}(d\varphi(X)^p),$$
is defined over $\Zp$, and is $G_{\Zp}$-equivariant.
\end{lemma}

\begin{proof}
The free $\Zp$-modules $\gZp$ and $\mathfrak{m}_{\Zp}$ are affine schemes over $\Zp$, with polynomial coordinate rings
$$\Zp[\gZp] \cong \Zp[x_1,\ldots,x_k], \qquad \Zp[\mathfrak{m}_{\Zp}] \cong \Zp[y_1,\ldots,y_{n^2-k}],$$
where $k$ is the rank of the free $\Zp$-module $\gZp$.

There is a composite of morphisms over $\Zp$,
$$\gZp \xrightarrow{d\varphi} \mathfrak{gl}_{n,\Zp} \xrightarrow{Y\mapsto Y^p} \mathfrak{gl}_{n,\Zp} \xrightarrow{\pi_{\mathfrak{m}}} \mathfrak{m}_{\Zp}.$$
Denote this composite as $\psi$.  Recalling the definition $m^p$, we observe that $\psi(X)=d\varphi(X)^p-d\varphi(m^p(X))$.  The morphism $\psi$ is equivalent to its comorphism
$$\psi^*: \Zp[y_1,\ldots,y_{n^2-k}] \rightarrow \Zp[x_1,\ldots,x_k].$$
Over $\mathbb{C}$, this sequence of morphisms can be further extended by the $\mathbb{C}$-linear map on $\mathfrak{m}_{\mathbb{C}}$,
$$\mathfrak{m}_{\mathbb{C}} \xrightarrow{Y \mapsto \frac{1}{p}Y} \mathfrak{m}_{\mathbb{C}}.$$
The composite of all of these morphisms is defined over $\mathbb{C}$, and to show it is defined over $\Zp$, we must show that the resulting composite of ring maps
$$\mathbb{C}[y_1,\ldots,y_{n^2-k}] \xrightarrow{y_i \mapsto \frac{y_i}{p}} \mathbb{C}[y_1,\ldots,y_{n^2-k}] \xrightarrow{\psi^*} \mathbb{C}[x_1,\ldots,x_k],$$
sends each $y_i$ to $\Zp[x_1,\ldots,x_k]$.  This is equivalent to checking that each $\psi^*(y_i)$ has image in $p\Zp[x_1,\ldots,x_k]$.

Working over $\Bbbk$, the morphism $\psi_{\Bbbk}$ sends $\g \rightarrow \{0\}$, precisely because $d\varphi(X)^p=d\varphi(X^{[p]})=d\varphi(m^p(X))$ here.  Since $\Bbbk$ is algebraically closed, it follows that the map   
$$\psi^*_{\Bbbk}: \Bbbk[y_1,\ldots,y_{n^2-k}] \rightarrow \Bbbk[x_1,\ldots,x_k],$$
satisfies $\psi^*_{\Bbbk}(y_i)=0$.  Since the reduction mod $p$ of the image of each $\psi^*(y_i)$ is $0$, it must be that $\psi^*(y_i) \in p\Zp[x_1,\ldots,x_k]$.

We now have a morphism defined over $\Zp$.  Since the schemes are all affine, to prove $G_{\Zp}$-equivariance, we must verify the commutativity of the diagram of $\Zp$-algebra homomorphisms
\[
  \begin{tikzcd}
    \Zp[\mathfrak{m}_{\Zp}] \arrow{r}{\sigma^*} \arrow{d} & \Zp[\gZp] \arrow{d}\\
    \Zp[G_{\Zp}] \otimes_{\Zp} \Zp[\mathfrak{m}_{\Zp}] \arrow{r}{id. \otimes \sigma^*} & \Zp[G_{\Zp}] \otimes_{\Zp} \Zp[\gZp].
  \end{tikzcd}
\]
where the two vertical maps encode the action of $G_{\Zp}$ on the two schemes.  Over $\mathbb{C}$, the morphism $$\sigma_{\mathbb{C}}: \mathfrak{g}_{\mathbb{C}} \rightarrow \mathfrak{m}_{\mathbb{C}}$$
is $G_{\mathbb{C}}$-equivariant, so that the diagram above commutes when tensored over $\Zp$ with $\mathbb{C}$.  We have then two $\Zp$-algebra maps from the polynomial ring $\Zp[\mathfrak{m}_{\Zp}]$ to the free $\Zp$-module $\Zp[G_{\Zp}] \otimes_{\Zp} \Zp[\gZp]$ that are the same when tensored over $\Zp$ with $\mathbb{C}$.  It follows that they are the same over $\Zp$, hence that the diagram above commutes.
\end{proof}

\begin{theorem}
The $B_{\mathbb{C}}$-equivariant morphism
$\tilde{E}_p: \mathfrak{u}_{\mathbb{C}} \rightarrow U_{\mathbb{C}}$
is defined over $\Zp$, and over $\Bbbk$ extends uniquely to a generalized exponential map for $G$.
\end{theorem}

\begin{proof}
We work inside the matrix representation for $G_{\Zp}$ from Table \ref{table:2}.  The previous lemma allows us to define a $G_{\Zp}$-equivariant morphism
$$\psi: d\varphi(\gZp) \rightarrow d\varphi(\gZp),$$
given by
$$\psi(X)=\pi_{\mathfrak{g}}\left(E_p(d\varphi(X))E_p\left(-\frac{\pi_{\mathfrak{m}}(d\varphi(X)^p)}{p}\right)\right).$$
Applying Lemma \ref{isregular}, this morphism restricts over $\mathbb{C}$ to a $G_{\mathbb{C}}$-equivariant automorphism of $d\varphi(\mathcal{N}(\gC))$.

Over $\mathbb{C}$ we have factorizations
$$E_p(d\varphi(X))=\exp(d\varphi(X))\exp\left(\frac{d\varphi(X)^p}{p}\right)\cdots,$$
and 
$$E_p\left(-\frac{\pi_{\mathfrak{m}}(d\varphi(X)^p)}{p}\right)=\exp\left(-\frac{\pi_{\mathfrak{m}}(d\varphi(X)^p)}{p}\right)\exp\left(\frac{1}{p}\left(-\frac{\pi_{\mathfrak{m}}(d\varphi(X)^p)}{p}\right)^p\right)\cdots.$$
We can also consider $\varphi(\tilde{E}_p(X))$ over $\mathbb{C}$.  We know that the exponential map over $\mathbb{C}$ preserves algebraic subgroups \cite[Proposition 7.1]{M}, so that for any $X \in \mathcal{N}(\gC)$, $\varphi(\exp(X))=\exp(d\varphi(X))$, where the first exponential is on $\mathcal{N}(\gC)$, and the second exponential is on $\mathcal{N}(\mathfrak{gl}_{n,\mathbb{C}})$.  Recalling also the definition of $m^i$, we have a factorization
\begin{align*}
\varphi(\tilde{E}_p(X))=& \varphi\left(\exp\left(X + \frac{m^p(X)}{p} + \cdots \right)\right)\\
=& \exp(d\varphi(X))\exp\left(\frac{d\varphi(m^p(X))}{p}\right)\cdots \\
=& \exp(d\varphi(X))\exp\left(\frac{\pi_{\mathfrak{g}}(d\varphi(X)^p)}{p}\right)\cdots .
\end{align*}
We claim that $\pi_{\mathfrak{g}} \circ \varphi \circ \tilde{E}_p = \psi$ as $B_{\mathbb{C}}$-equivariant automorphisms of $d\varphi(\mathfrak{u}_{\mathbb{C}})$.  Choose $\theta$ to be an associated cocharacter of $X$, so that $\varphi \circ \theta$ is a cocharacter of $GL_{n,\mathbb{C}}$.  Both $\pi_{\mathfrak{g}} \circ \varphi \circ \tilde{E}_p(X)$ and $\psi(X)$ consist of a sum of $\varphi(\theta(\mathbb{G}_{m,\mathbb{C}}))$ weight vectors in $d\varphi(C_{\gC}(X))$.  By Lemma \ref{weightsofcenter} it follows that $\pi_{\mathfrak{g}}$ is zero on an element $Y \in \mathfrak{gl}_{n,\mathbb{C}}$ if
$$Y \in C_{\mathfrak{gl}_{n,\mathbb{C}}}(d\varphi(X)), \quad Y \in \mathfrak{gl}_{n,\mathbb{C}}(\varphi \circ \theta; z),$$
and $z$ is not an exponent of the Weyl group of $G$.  In particular, from Table \ref{table:1} we see that $z$ is not an exponent of the Weyl group whenever $z \ge h$.

The factorizations above show that
\begin{equation}\label{inlargeweightspaces}
\varphi(\tilde{E}_p(X))-\psi(X) \in \sum_{j \ge p^2} \mathfrak{gl}_{n,\mathbb{C}}(\varphi \circ \theta; j).
\end{equation}
Because $p$ is a good prime for the root system of $G$, and that root system is of exceptional type, we have $p^2>h$.  Thus, (\ref{inlargeweightspaces}) implies that $\varphi(\tilde{E}_p(X))$ and $\psi(X)$ have the same projection to $d\varphi(\gC)$ as the terms in their difference are all mapped to $0$.  By identifying $U_{\mathbb{C}}$ and $\mathfrak{u}_{\mathbb{C}}$ with their embedded images in $GL_{n,\mathbb{C}}$ and $\mathfrak{gl}_{n,\mathbb{C}}$ respectively, and with $\pi^{-1}$ being the inverse of the isomorphism in Lemma \ref{projectionisinvertible}, we conclude that
$$\tilde{E}_p = \pi^{-1} \circ \psi.$$
As the morphisms on the right are both defined over $\Zp$, it stands that $\tilde{E}_p$ is also.

Finally, we show that this map defines an embedding of truncated Witt groups into $U_{\Zp}$.  Let $X$ continue to be regular nilpotent.  Assume that $p < h$ but is still a good prime (the case when $p \ge h$ is immediate).  There is a morphism of schemes over $\Zp$,
$$f:\mathcal{W}_{2,\Zp} \rightarrow U_{\Zp},$$
given by
$$f(a_1,a_0)=\tilde{E}_p(a_1X)\tilde{E}_p(a_0m^p(X)).$$
It is an algebraic group homomorphism, by observing that over $\mathbb{C}$,
\begin{align*}
f(b_1,b_0)f(a_1,a_0)=&\tilde{E}_p(b_1X)\tilde{E}_p(b_0m^p(X))\tilde{E}_p(a_1X)\tilde{E}_p(a_0m^p(X))\\
=&\exp((b_1+a_1)X)\exp\left((b_1^p + a_1^p)\frac{m^p(X)}{p}\right)\exp\left((b_0+a_0)m^p(X)\right)\\
=&\exp((b_1+a_1)X)\exp\left(\left(b_0+a_0+\frac{(b_1^p + a_1^p)}{p}\right)m^p(X)\right).\\
\end{align*}
On the other hand,
\begin{align*}
f((b_1,b_0)+(a_1,a_0))=&f\left(b_1+a_1,b_0+a_0 + \frac{b_1^p+a_1^p-(b_1 + a_1)^p}{p}\right)\\
=&\exp((b_1+a_1)X)\exp\left((b_1+a_1)^p\frac{m^p(X)}{p}\right)\\
& \times \exp\left(b_0+a_0 + \frac{b_1^p+a_1^p-(b_1 + a_1)^p}{p}m^p(X)\right)\\
=&\exp((b_1+a_1)X)\exp\left(\left(b_0+a_0+\frac{(b_1^p + a_1^p)}{p}\right)m^p(X)\right).\\
\end{align*}
We see then that this map respects the group structure of $\mathcal{W}_{2,\mathbb{C}}$.  By looking at the comorphisms, arguments similar to those used earlier and the fact that $\Zp[\mathcal{W}_{2,\Zp}]$ is free over $\Zp$ shows that this is a group homormophism over $\Zp$.  Working over $\Bbbk$, the fact that $\tilde{E}_p$ is a Springer isomorphism can be used to show that this morphism sends $\mathcal{W}_{2}$ isomorphically to its image.

\end{proof}

\section{Decomposing The Connected Center of $C_G(X)$}

Let $X \in \mathcal{N}(\mathfrak{g})$.  We will show that any generalized exponential map can be used to obtain a decomposition of $Z(C_G(X))^0$ into a direct product of truncated Witt groups.  This product holds as algebraic varieties, strengthening the result of Seitz in \cite[Theorem 1]{Sei2}, where such a decomposition is observed to hold as abstract groups.  We further prove that when $G$ is not of type $D_{p^m+1}$ for any $m>0$, then there is canonical truncated Witt group embedded in $G$ that contains $X$ in its Lie algebra.  This fact will feature prominently in the next section when we attempt to classify all generalized exponential maps for $G$.

\subsection{Decomposing $Z(C_G(X))^0$}
Let $\phi$ be a generalized exponential map.  Let $X \in \mathcal{N}(\g)$, and $\theta$ an associated cocharacter of $X$.  It is known that $Z(C_G(X)^0)$ is unipotent; this result is attributed in \cite{LT} to Proud, who proved it in an unpubished manuscript \cite{Pr2}.  Since this group is also connected and abelian, then by general theory it is at least isogenous to a product of truncated Witt groups (cf. \cite[VII.2, Theorem 1]{Ser2}).  We will show that $Z(C_G(X)^0)$ is isomorphic to such a product.

The image of $\theta$ normalizes $Z(C_G(X)^0)$, thus we can find a basis $\{e_1,e_2,\ldots,e_{\ell}\}$ of vectors in $\text{Lie}(Z(C_G(X)))$ such that $e_i \in \g(\theta,j_i)$ for some $j_i$.  In this case, it happens that each $j_i > 0$ (c.f. \cite[Proposition 5.10(b)]{J2}).  Now, refine this set so that it is no longer necessarily a basis, but that it, together with all non-zero $[p^k]$-th powers of the elements therein for all $k > 0$, form a basis for $\text{Lie}(Z(C_G(X)))$.  Let $\mathcal{W}_{(\phi,e_i)}$ be the closed subgroup of $G$ that is the image of the morphism $\phi_{e_i}$ from Definition \ref{embeddingproperty}.

\begin{remark}\label{sameliesubalgebra}
It follows by the construction of the morphism $\phi_{e_i}$ that $\text{Lie}(\mathcal{W}_{(\phi,e_i)})$ is precisely the $[p]$-restricted Lie subalgebra of $\g$ that is generated by $e_i$.
\end{remark}

\begin{prop}\label{centralizerdecomp}
Let $X \in \Ng$, and let the subgroups $\mathcal{W}_{(\phi,e_i)}$ be defined as above.  Each $\mathcal{W}_{(\phi,e_i)} \subseteq Z(C_G(X))^0$, and there is an isomorphism of algebraic groups
$$\mathcal{W}_{(\phi,e_1)} \times \mathcal{W}_{(\phi,e_2)} \times \cdots \times \mathcal{W}_{(\phi,e_{\ell})} \cong Z(C_G(X))^0.$$
\end{prop}

\begin{proof}
The argument near the end of the proof of \cite[Theorem 4.2]{MT}, applied to the Springer isomorphism $\phi$, shows that $\phi$ maps $\text{Lie}(Z(C_G(X)))$ isomorphically onto $Z(C_G(X))^0$.   Since $e_i \in \text{Lie}(Z(C_G(X)))$, it then follows that $\mathcal{W}_{(\phi,e_i)} \subseteq Z(C_G(X))^0$.  Inclusion, followed by repeated multiplication, defines a homomorphism of algebraic groups
$$f: \mathcal{W}_{(\phi,e_1)} \times \mathcal{W}_{(\phi,e_2)} \times \cdots \times \mathcal{W}_{(\phi,e_{\ell})} \rightarrow Z(C_G(X))^0.$$
The differential $df$ must be an isomorphism of Lie algebras, as it is surjective (the image contains each ${e_i}^{[p^m]}$), and the respective Lie algebras have the same dimension.  Therefore the kernel of $f$ has dimension $0$.  On the other hand, $\theta(\mathbb{G}_m)$ acts on each $\mathcal{W}_{(\phi,e_i)}$ by conjugation, and therefore can be given a component-wise action on the product of these groups.  Further, the homomorphism $f$ is equivariant with respect to this action, so that the kernel of $f$ is stable under the action of $\theta(\mathbb{G}_m)$.  As it is $0$-dimensional, it must in fact consist of fixed points for $\theta(\mathbb{G}_m)$.  But each $e_i$ is a non-zero weight vector for $\theta(\mathbb{G}_m)$ (as noted above), so there are no non-zero $\theta(\mathbb{G}_m)$-fixed points in $\text{Lie}(Z(C_G(X)))$.  Consequently, the only $\theta(\mathbb{G}_m)$-fixed point in any $\mathcal{W}_{(\phi,e_i)}$ is the identity element, and by extension the same is true for the action on the product of these groups.  This shows that $f$ is an injective map on $\Bbbk$-points.  But the image of $f$ is closed and connected (it is equal to the subgroup generated by all of the $\mathcal{W}_{(\phi,e_i)}$ which are each closed and connected), and since $Z(C_G(X))^0$ is connected of the same dimension, $f$ is also surjective on $\Bbbk$-points.  It is therefore a bijective morphism on $\Bbbk$-points having bijective differential, hence an isomorphism of algebraic groups.
\end{proof}

\subsection{Truncated Witt Groups Over Nilpotent Elements}

\begin{theorem}\label{canonicalnilpotentovergroup}
Suppose that $\phi$ and $\psi$ are two generalized exponential maps for $G$, and that the root system of $G$ is not of type $D_{p^n+1}$ for any $n>0$.  Then $\mathcal{W}_{(\phi,X)}=\mathcal{W}_{(\psi,X)}$ for every $X \in \Ng$.
\end{theorem}

\begin{proof}
Suppose first that $X$ is a regular nilpotent element, with associated cocharacter $\theta$.  In this case, $Z(C_G(X))^0=C_G(X)^0$, and $\text{Lie}(Z(C_G(X))^0)=C_{\g}(X)$.  Choose a set of generating elements of $C_{\g}(X)$ with respect to $\theta$ as in the setup prior to Proposition \ref{centralizerdecomp}, and order $\{e_1,\ldots,e_{\ell}\}$ so that $X=e_1$.  By the proposition, we have
$$C_G(X)^0 \cong \mathcal{W}_{(\phi,X)} \times \mathcal{W}_{(\phi,e_2)} \times \cdots \times \mathcal{W}_{(\phi,e_{\ell})},$$
and also that
$$C_G(X)^0 \cong \mathcal{W}_{(\psi,X)} \times \mathcal{W}_{(\psi,e_2)} \times \cdots \times \mathcal{W}_{(\psi,e_{\ell})}.$$
Also, as in Remark \ref{sameliesubalgebra},
$$\text{Lie}(\mathcal{W}_{(\psi,X)}) = \text{Span}\{X,X^{[p]},\ldots,X^{[p^{m-1}]}\} = \text{Lie}(\mathcal{W}_{(\phi,X)}).$$
Let $f$ be the composite of morphisms
$$\mathcal{W}_{(\psi,X)} \hookrightarrow C_G(X)^0 \rightarrow C_G(X)^0/\mathcal{W}_{(\phi,X)}.$$
Because of the direct product decomposition of $C_G(X)^0$ given above, it follows that
$$\text{Lie}(C_G(X)^0/\mathcal{W}_{(\phi,X)}) \cong C_{\mathfrak{g}}(X)/\text{Lie}(\mathcal{W}_{(\phi,X)}).$$
The equality of the Lie algebras of $\mathcal{W}_{(\phi,X)}$ and $\mathcal{W}_{(\psi,X)}$ show that $df=0$.  Additionally, by Lemma \ref{weightsofcenter}, and the table of exponents of the Weyl group, there are no remaing $\theta(\mathbb{G}_m)$-weight vectors in $C_{\mathfrak{g}}(X)/\text{Lie}(\mathcal{W}_{(\phi,X)})$ having weight $2p^i$ for any $i\ge 0$ (since we are assuming that we are not in type $D_{p^n+1}$, all vectors having such a weight in $C_{\g}(X)$ are of the form $X^{[p^i]}$).

We claim that the last fact will imply that $f$ is trivial on $\Bbbk$-points.  To see this, let $p^c$ be the maximal order of an element in $f(\mathcal{W}_{(\psi,X)})$.  Consider then the morphism
$$f: \mathcal{W}_{(\psi,X)} \rightarrow H,$$
where $H$ is the subquotient group of elements of order $p^c$ in $C_G(X)^0/\mathcal{W}_{(\phi,X)})$ modulo the subgroup of elements of order less than $p^c$.  Every non-zero element in $H$ has order $p$, so that $H$ is a vector group, and we can view $f$ as a morphism
$$f: \mathcal{W}_{(\psi,X)} \rightarrow \mathbb{G}_a \times \cdots \times \mathbb{G}_a,$$
and this map is still $\theta(\mathbb{G}_m)$-equivariant.  This morphism must be trivial on the subgroup ${\mathcal{W}_{(\psi,X)}}^p$ of all $p$-th powers in $\mathcal{W}_{(\psi,X)}$, therefore the morphism factors through the quotient
$$\mathcal{W}_{(\psi,X)}/{\mathcal{W}_{(\psi,X)}}^p \cong \mathbb{G}_a.$$  The action of $\theta(\mathbb{G}_m)$ on this quotient is by weight $2$ (since $X \in \g(\theta;2)$).  If this morphism above is non-trivial, then the $\theta(\mathbb{G}_m)$-equivariance forces one of the factors on the right to have a $\theta(\mathbb{G}_m)$-action by the cocharacter $2p^j$ for some $j$.  But if this held, then the Lie algebra of the group on the right would contain a $\theta(\mathbb{G}_m)$-weight vector of weight $2p^j$.  This cannot happen since this would result in a weight vector of the same weight in $C_{\mathfrak{g}}(X)/\text{Lie}(\mathcal{W}_{(\phi,X)})$.  Therefore $f$ must be the trivial homomorphism on $\Bbbk$-points, so that $\psi_X(\mathcal{W}_m) \subseteq \phi_X(\mathcal{W}_m)$.  By connectedness and equality of dimension, this containment is then an equality.

We next show that the same result holds for an arbitrary nilpotent element $Y$.  Since $\mathcal{W}_{(\psi,X)} = \mathcal{W}_{(\phi,X)}$, then by the construction of these closed subgroups, there are elements $c_0,\ldots,c_{m-1}$ in $\Bbbk$, $c_0 \ne 0$, such that
\begin{equation}\label{thisisthat}\phi(X)=\psi(c_0X)\psi(c_1X^{[p]})\cdots \psi(c_{m-1}X^{[p^{m-1}]}).\end{equation}
If we can show that this relation holds for all nilpotent elements, then that will prove the claim.  To do so, it suffices to show that the morphism
$$\mathcal{N}(\g) \rightarrow \mathcal{U}(G), \qquad Y \mapsto \psi(c_0Y)\psi(c_1Y^{[p]})\cdots \psi(c_{m-1}Y^{[p^{m-1}]}),$$
is $G$-equivariant.  But this is clear as the $[p]$-th power map, scalar multiplication, and $\psi$ are all $G$-equivariant.  This morphism is therefore equal to $\phi$ on the regular orbit, hence on all of $\mathcal{N}(\g)$ by density.
\end{proof}

\section{Parameterizing Generalized Exponential Maps}\label{uniqueness}

In this section we will parameterize all generalized exponential maps for $G$.  We begin with an example.

\begin{example}
Let $\Bbbk$ have characteristic $2$.  The morphisms $\phi(X)=1+X+X^2$ and $\psi(X)=1+X+X^3$ each define generalized exponential maps for $GL_4$.  In fact, $\phi(X)=E_2(X)$ (reducing the coefficients in $E_2(t)$ mod $2$).  The two morphisms are related by the equation
$$\psi(X)=\phi(X)\phi(X^2).$$
In the notation of restricted Lie algebras, this can be expressed as
$$\psi(X)=\phi(X)\phi(X^{[2]}).$$
For all $a,b \in \Bbbk$, we have
$$\psi(aX)\psi(bX^{[2]})=\phi(aX)\phi(a^2X^2)\phi(bX^2),$$
from which it can be worked out that $\mathcal{W}_{(\phi,X)}=\mathcal{W}_{(\psi,X)}$.
\end{example}

Using the results in the previous section, we can now generalize the relationship between $\phi$ and $\psi$ in general.

\begin{theorem}
Let $p^m$ be the nilpotence degree of a regular nilpotent element in $\g$.  Then set of all generalized exponential maps for $G$ is in bijection with the set:
$$\begin{array}{ll}
\bullet \; \Bbbk^{\times} & \textup{if } m=1.\\
\bullet \; \mathbb{F}_p^{\times} \times {\mathbb{F}_p}^{\times m-2} \times \Bbbk & \textup{if } m > 1 \textup{ and } G \textup{ not of type } D_{p^n+1} \textup{ for any } n>0.\\
\bullet \; \mathbb{F}_p^{\times} \times {\mathbb{F}_p}^{\times m-2} \times \Bbbk^{\times 2} \quad & \textup{if } m > 1 \textup{ and } G \textup{ is of type } D_{p^n+1} \textup{ for some } n>0.\\
\end{array}$$
\end{theorem}

\begin{proof}
The first case occurs only when $p \ge h$, and thus is effectively saying that the exponential map is the unique generalized exponential map, up to precomposing with a non-zero scaling of $\Ng$.  We will therefore focus on the other two cases.

Let $\phi$ be a fixed generalized exponential map for $G$, and suppose that $\psi$ is any other generalized exponential map.  If $G$ is not of type $D_{p^n+1}$, then by \ref{canonicalnilpotentovergroup} we have that $\psi(X) \in \mathcal{W}_{(\phi,X)}$.  There are therefore elements $a_0,\ldots,a_{m-1} \in \Bbbk$ such that
$$\psi(X)=\phi(a_0X)\phi(a_1X^{[p]})\cdots \phi(a_{m-1}X^{[p^{m-1}]}).$$
Now, the mapping $Y \mapsto \phi(a_0Y)\phi(a_1Y^{[p]})\cdots \phi(a_{m-1}Y^{[p^{m-1}]})$ is easily checked to be $G$-equivariant, hence it follows that for every $Y \in \Ng$,  
$$\psi(Y)=\phi(a_0Y)\phi(a_1Y^{[p]})\cdots \phi(a_{m-1}Y^{[p^{m-1}]}).$$
Because $\psi$ is a generalized exponential map, we have
$$\psi(X^{[p^i]})=\psi(X)^{p^i}$$
for all $i>0$.  This then forces an equality
$$\phi(a_0X^{[p]})\cdots \phi(a_{m-2}X^{[p^{m-1}]})=\phi({a_0}^pX^{[p]})\cdots \phi({a_{m-2}}^pX^{[p^{m-1}]})$$
from which it follows that $a_0,\ldots,a_{m-2} \in \mathbb{F}_p$.  Of course, $a_0$ also must be non-zero for $\psi$ to be an isomorphism.  On the other hand, one can easily check that for any sequence $a_0,\ldots,a_{m-1}$ satisfying these conditions, the map
$$\psi(X)=\phi(a_0X)\phi(a_1X^{[p]})\cdots \phi(a_{m-1}X^{[p^{m-1}]})$$
is a generalized exponential map for $G$.

Suppose now that $G$ is of type $D_{p^n+1}$ for some $n>0$.  Fix $\theta$ an associated cocharacter of $X$ and a basis for $C_{\g}(X)$ as in Proposition \ref{weightsofcenter}.  Here, $r = p^n+1$, and $X_r$ is a $\theta(\mathbb{G}_m)$-weight vector of weight $2p^n$ that is linearly independent from $X^{[p^n]}$.  We claim that this case is different from the previous case as there is now an extra parameter involved.  That is, an arbitrary generalized exponential map $\psi$ is related to $\phi$ by the equation
$$\psi(X)=\phi(a_0X)\phi(a_1X^{[p]})\cdots \phi(a_{m-1}X^{[p^{m-1}]})\phi(b\gamma_r(X)),$$
with $\gamma_r$ as in Lemma \ref{likemultilinear}.  By the statement of that lemma, $\gamma_r(X^{[p^i]})$ is a $\theta(\mathbb{G}_m)$-weight vector of weight $2p^ip^n=2p^{i+n}$, and by $G$-equivariance it follows that $C_G(X)$ centralizes $\gamma_r(X^{[p^i]})$ so that $\gamma_r(X^{[p^i]}) \in C_\g(X)$.  But $p^{i+n} > 2(p^n+1)-3$ if $i > 0$, so by Table \ref{table:1} and Proposition \ref{weightsofcenter} we have that $\gamma_r(X^{[p^i]}) = 0$ if $i > 0$.  The same $\theta(\mathbb{G}_m)$ weight space considerations show that $X_r^{[p]} = 0$ also.  Finally, every Springer isomorphism (in characteristic $p$) sends $[p]$-nilpotent elements to $p$-unipotent elements, so that $\phi(b\gamma_r(X))^p=1$.  From this we verify that $\psi(X)^p=\psi(X^{[p]})$, and as before it can be verified directly that $\psi$ is a generalized exponential map.
\end{proof}

\begin{remark}
If we pick a $\Zp$ group scheme $G_{\Zp}$ that base changes to $G$ as before, then we have shown that we can find a generalized exponential map $\phi$ that arises over $\Zp$, so in particular arises over $\Fp$.  We could therefore further restrict to only considering those $\psi$ that also arise over $\Fp$, in which case we would replace all occurrences of $\Bbbk$ in the previous theorem with $\Fp$ (evidently there can be only finitely many such maps arising over $G_{\Fp}$).
\end{remark}

\section{When the characteristic is not separably good}

Let $G_{sc}$ be the simply connected group that is isogenous to $G$.  We conclude with a look at what happens in good characteristic when the covering map $G_{sc} \rightarrow G$ is not separable, which is a type $A$ phenomenon only.  The issues here are more subtle than when the characteristic is bad, as in that case there cannot even be a $G$-equivariant bijection due to the differing number of $G$-orbits in the respective varieties.

Denote by $\text{Pr}$ the covering map.  It induces $G$-equivariant bijective morphisms,
\begin{equation}\label{inducedmaps}
\text{pr}: \mathcal{N}(\mathfrak{g}_{sc}) \rightarrow \mathcal{N}(\mathfrak{g}) \quad \text{and} \quad \text{Pr}:\mathcal{U}(G_{sc}) \rightarrow \mathcal{U}(G).
\end{equation}
Because these maps are both bijections on the respective $\Bbbk$-points, any Springer isomorphism $\phi$ for $G_{sc}$ induces a $G$-equivariant bijective map $\overline{\phi}$ from $\mathcal{N}(\mathfrak{g})$ to $\mathcal{U}(G)$. When the covering map is separable, the maps in (\ref{inducedmaps}) are indeed isomorphisms, and one can apply an obvious argument to see that Springer isomorphisms for $G_{sc}$ yield Springer isomorphisms for $G$ (and vice-versa).  Even when the covering map is not separable, one can still show that $\overline{\phi}$ satisfies some nice properties.

The following was pointed out to us by J.-P. Serre.

\begin{prop}
Any Springer isomorphism $\phi$ for $G_{sc}$ induces a map
$$\overline{\phi}: \mathcal{N}(\mathfrak{g}) \rightarrow \mathcal{U}(G)$$
having the following properties:
\begin{enumerate}
\item it is a $G$-equivariant homeomorphism. 
\item its graph is a closed irreducible subvariety of $\mathcal{N}(\mathfrak{g}) \times \mathcal{U}(G)$.
\end{enumerate}

\end{prop}

\begin{proof}
(1) The $G$-equivariance is clear, and that it is a homeomorphism is a consequence of the bijectivity of the maps in (\ref{inducedmaps}) and Zariski's Main Theorem.  (2)  The graph of $\phi$ is a closed subvariety of $\mathcal{N}(\g_{sc}) \times \mathcal{U}(G_{sc})$ (since $\mathcal{N}(\g_{sc})$ and $\mathcal{U}(G_{sc})$ are isomorphic).  Now, the map $$\text{pr} \times \text{Pr}: \mathcal{N}(\g_{sc}) \times \mathcal{U}(G_{sc}) \rightarrow \Ng \times \mathcal{U}(G)$$
is also a homeomorphism, and by construction sends the graph of $\phi$ to the graph of $\overline{\phi}$.  The graph of $\overline{\phi}$ is then a closed subset, hence a closed subvariety, and it is irreducible since the graph of $\phi$ is irreducible, a topological property, and the map $\text{pr} \times \text{Pr}$ is a homeomorphism.
\end{proof}

We now will show that in general it is not an isomorphism of varieties, by looking specifically at the case of $PGL_2$.  The following is a variation on a result communicated to us by Jay Taylor.

\begin{theorem}
There exists is a Springer isomorphism for $PGL_2$ if and only if $\overline{\phi}$ is such an isomorphism, where $\phi$ is any Springer isomorphism for $SL_2$.
\end{theorem}

\begin{proof}
Suppose there exists a Springer isomorphism for $PGL_2$.  Then this map is in particular a $PGL_2$-equivariant homeomorphism between $\mathcal{N}(\mathfrak{pgl}_2)$ and $\mathcal{U}(PGL_2)$.  This is equivalent to an $SL_2$-equivariant homeomorphism from $\mathcal{N}(\mathfrak{sl}_2)$ to $\mathcal{U}(SL_2)$, which is necessarily a Springer isomorphism for $SL_2$.  Now, any two Springer isomorphisms for $SL_2$ can only differ by precomposing one with scalar multiplication by some $c \in \Bbbk^{\times}$ on $\mathcal{N}(\mathfrak{sl}_2)$.  Of course, scalar multiplication by $c$ on $\mathcal{N}(\mathfrak{pgl}_2)$ is a $PGL_2$-equivariant isomorphism, and is induced by the corresponding map on $\mathcal{N}(\mathfrak{sl}_2)$.  It follows that every Springer map for $SL_2$ then pushes down to a composite of two $PGL_2$-equivariant isomorphisms, so that every Springer isomorphism for $SL_2$ induces one for $PGL_2$.  This proves one direction of the claim in the theorem; the other direction is clear.
\end{proof}

\begin{cor}
There does not exist a Springer isomorphism for $PGL_2$ in characteristic $2$.
\end{cor}

\begin{proof}
This follows from the previous theorem, and the computation by Jay Taylor found in the appendix which shows that a particular Springer isomorphism for $SL_2$ pushes down to a map that is not a variety morphism.
\end{proof}

\begin{remark}
As alluded to above, the note shared with us by Taylor contained not only the computation that follows in the Appendix, but also a proof of the previous corollary, using similar reasoning to that given here.
\end{remark}

\section{Appendix: Computing $\overline{\phi}$ For $PGL_2$ (Jay Taylor)}

Let $\Bbbk$ have characteristic $2$, and fix the usual ordered basis $\{E,H,F\} \subseteq \mathfrak{sl}_2$, namely
\begin{equation*}
E = \begin{bmatrix}
0 & 1\\
0 & 0
\end{bmatrix}
\qquad
H = \begin{bmatrix}
1 & 0\\
0 & 1
\end{bmatrix}
\qquad
F = \begin{bmatrix}
0 & 0\\
1 & 0
\end{bmatrix}.
\end{equation*}

The choice of basis $\{E,H,F\}$ gives an identification of $GL(\mathfrak{sl}_2)$ with $GL_3$. Hence we may, and will, consider $PGL_2$ as a closed subgroup of $GL_3$, and $\mathfrak{pgl}_2$ as a subalgebra of $\mathfrak{gl}_3$.  The following is a straightforward computation left to the reader.

\begin{lemma}\label{lem:Ad-desc}
With respect to the basis $\{E,H,F\} \subseteq \mathfrak{sl}_2$ we have
\begin{equation*}
\Ad \begin{bmatrix}
x_1 & x_2\\
x_3 & x_4
\end{bmatrix}
=
\begin{bmatrix}
x_1^2 & 0 & x_2^2\\
x_1x_3 & 1 & x_2x_4\\
x_3^2 & 0 & x_4^2
\end{bmatrix}
\end{equation*}
for any $\displaystyle \left[\begin{smallmatrix}x_1 & x_2\\ x_3 & x_4 \end{smallmatrix}\right] \in SL_2$.
\end{lemma}

One readily checks that the image of $\Ad$ coincides with the following closed set
\begin{equation}\label{eq:G-def}
PGL_2 = \left\{\left.
\begin{bmatrix}
x_1 & 0 & x_2\\
x_3 & 1 & x_4\\
x_5 & 0 & x_6
\end{bmatrix}
\in \mathfrak{gl}_3
\,\right|\,
\begin{gathered}
0 = x_1x_5+x_3^2\\ 0 = x_2x_6+x_4^2\\ 0 = x_1x_6+x_2x_5+1
\end{gathered}
\right\}.
\end{equation}
It will also be useful to note that the inverse $\Ad^{-1} : PGL_2 \to SL_2$ is given by
\begin{equation}\label{eq:Ad-inv}
\Ad^{-1}\begin{bmatrix}
x_1 & 0 & x_2\\
x_3 & 1 & x_4\\
x_5 & 0 & x_6
\end{bmatrix}
=
\begin{bmatrix}
\sqrt{x_1} & \sqrt{x_2}\\
\sqrt{x_5} & \sqrt{x_6}\\
\end{bmatrix}.
\end{equation}

The unipotent variety $\mathcal{U}(SL_2)$ is given by
\begin{equation*}
\mathcal{U}(SL_2) = \left\{\left.\begin{bmatrix}x_1&x_2\\x_3&x_1\end{bmatrix} \in \mathfrak{gl}_2 \,\right|\, x_1^2+x_2x_3 = 1\right\}.
\end{equation*}
The bijective morphism $\Ad : SL_2 \to PGL_2$ restricts to a bijective morphism $\Ad : \mathcal{U}(SL_2) \to \mathcal{U}(PGL_2)$. From \cref{eq:G-def} we get the following.

\begin{lemma}\label{lem:unipotent-variety}
$\mathcal{U}(PGL_2)$ is given by the closed subset
\begin{equation*}
\left\{\left.\begin{bmatrix}
x_1 & 0 & x_2\\
x_3 & 1 & x_4\\
x_5 & 0 & x_1
\end{bmatrix}
\in \mathfrak{gl}_3
\,\right|\,
\begin{gathered}
0 = x_1x_5+x_3^2\\
0 = x_1x_2+x_4^2\\
0 = x_1^2+x_2x_5+1
\end{gathered}
\right\}.
\end{equation*}
\end{lemma}

\begin{proof}
The characteristic polynomial of a matrix in \cref{eq:G-def} is
\begin{equation*}
(1-t)(t^2+(x_1+x_6)t+x_2x_5+x_1x_6).
\end{equation*}
The matrix is unipotent if and only if this characteristic polynomial is $(t-1)^3$. Comparing these polynomials gives the result.
\end{proof}

To obtain the Lie algebra $\mathfrak{pgl}_2$ we compute the locus of the differentials of the defining polynomials given in \cref{eq:G-def}. One obtains a $3$-dimensional space containing $\mathfrak{pgl}_2$, which must be $\mathfrak{pgl}_2$ since $\dim\mathfrak{pgl}_2 = 3$. Specifically, we get that
\begin{equation}
\mathfrak{pgl}_2 = \left\{
\begin{bmatrix}
x_1 & 0 & 0\\
x_2 & 0 & x_3\\
0 & 0 & x_1
\end{bmatrix}
\in \mathfrak{gl}_3
\right\}.
\end{equation}

Let $\ad : \mathfrak{sl}_2 \to \mathfrak{gl}_3$ be the adjoint representation, i.e., the differential of $\Ad$.

\begin{lemma}\label{lem:ad-desc}
With respect to the basis $\{E,H,F\} \subseteq \mathfrak{sl}_2$ we have
\begin{equation*}
\ad \begin{bmatrix}
x_1 & x_2\\
x_3 & x_1
\end{bmatrix}
=
\begin{bmatrix}
0 & 0 & 0\\
x_3 & 0 & x_2\\
0 & 0 & 0
\end{bmatrix}
\end{equation*}
for any $\left[\begin{smallmatrix}x_1 & x_2\\ x_3 & x_1 \end{smallmatrix}\right] \in \mathfrak{sl}_2$.
\end{lemma}

\begin{lemma}
The nilpotent cone $\mathcal{N}(\mathfrak{pgl}_2) = [\mathfrak{pgl}_2,\mathfrak{pgl}_2]$, and is therefore a linear subspace of $\mathfrak{pgl}_2$, isomorphic as a variety to $\mathbb{A}^2$.
\end{lemma}

Let $\phi$ be the Springer isomorphism for $SL_2$ given by $\phi(X) = I_2+X$.  One easily obtains the following.

\begin{lemma}
The maps $\overline{\phi}: \mathcal{N}(\mathfrak{pgl}_2) \to \mathcal{U}(PGL_2)$ and $(\overline{\phi})^{-1} : \mathcal{U}(PGL_2) \to \mathcal{N}(\mathfrak{pgl}_2)$ are given by
\begin{align*}
\overline{\phi}\begin{bmatrix}
0 & 0 & 0\\
x & 0 & y\\
0 & 0 & 0
\end{bmatrix}
&=
\begin{bmatrix}
xy+1 & 0 & y^2\\
x(\sqrt{xy}+1) & 1 & y(\sqrt{xy}+1)\\
x^2 & 0 & xy+1
\end{bmatrix}\\
(\overline{\phi})^{-1}
\begin{bmatrix}
x_1 & 0 & x_2\\
x_3 & 1 & x_4\\
x_5 & 0 & x_1
\end{bmatrix}
&=
\begin{bmatrix}
0 & 0 & 0\\
\sqrt{x_5} & 0 & \sqrt{x_2}\\
0 & 0 & 0
\end{bmatrix}.
\end{align*}
Hence neither map is a morphism of varieties.
\end{lemma}

\end{document}